\documentclass[11pt]{article}

\usepackage{amsmath}
\usepackage{amsthm}
\usepackage{amsfonts}
\usepackage{bbm}

\newcommand{\me}{\mathbb{E}}
\newcommand{\mr}{\mathbb{R}}

\newcommand{\mn}{\mathbb{N}}
\newcommand{\mmp}{\mathbb{P}}
\newcommand{\tb}{\texttt{b}}
\newcommand{\mb}{\mathfrak{b}}
\DeclareMathOperator{\1}{\mathbbm{1}}

\newtheorem{thm}{Theorem}[section]
\newtheorem{lemma}[thm]{Lemma}

\newtheorem{cor}[thm]{Corollary}

\newtheorem{assertion}[thm]{Proposition}
\theoremstyle{definition}

\theoremstyle{remark}
\newtheorem{rem}[thm]{Remark}

\begin{document}

\title{Limit theorems for discounted convergent perpetuities}\date{}
\author{Alexander Iksanov\footnote{Faculty of Computer Science and Cybernetics, Taras Shevchenko National University of Kyiv, Ukraine; e-mail address:
iksan@univ.kiev.ua} \ \ \ Anatolii Nikitin\footnote{Faculty of Natural Sciences, Jan Kochanowski University of Kielce; e-mail address: anatolii.nikitin@ujk.edu.pl} \ \ \  Igor Samoilenko\footnote{Faculty of Computer Science and Cybernetics, Taras Shevchenko National University of Kyiv, Ukraine; e-mail address: isamoil@i.ua}}  

\maketitle
\begin{abstract}
\noindent Let $(\xi_1, \eta_1)$, $(\xi_2, \eta_2),\ldots$ be independent identically distributed $\mathbb{R}^2$-valued random vectors. We prove a strong law of large numbers, a functional central limit theorem and a law of the iterated logarithm for convergent perpetuities $\sum_{k\geq 0}b^{\xi_1+\ldots+\xi_k}\eta_{k+1}$ as $b\to 1-$. Under the standard actuarial interpretation, these results correspond to the situation when the actuarial market is close to the customer-friendly scenario of no risk.
\end{abstract}

\noindent Key words: cluster set; functional central limit theorem; law of the iterated logarithm; perpetuity; strong law of large numbers

\noindent 2000 Mathematics Subject Classification: Primary: 60F15,60F17 \\
\hphantom{2000 Mathematics Subject Classification: } Secondary: 60G50

\section{Introduction}

Let $(\xi_1, \eta_1)$, $(\xi_2,\eta_2),\ldots$ be independent copies of an $\mr^2$-valued random vector $(\xi,\eta)$ with arbitrarily dependent components. Denote by $(S_k)_{k\in\mn_0}$ (as usual, $\mn_0:=\mn\cup\{0\}$) the standard random walk with jumps $\xi_k$ defined by $S_0:=0$ and $S_k:=\xi_1+\ldots+\xi_k$ for $k\in\mn$. Whenever a random series $\sum_{k\geq 0}e^{-S_k}\eta_{k+1}$ converges a.s., its sum is called {\it perpetuity} because of the following actuarial application. Assuming, for the time being, that $\xi$ and $\eta$ are a.s.\ positive, we can interpret $\eta_k$ and $e^{-\xi_k}$ as the planned payment and the discount factor (risk) for year $k$, respectively. Then $\sum_{k\geq 0}e^{-S_k}\eta_{k+1}$ can be thought of as `the present value of a permanent commitment to make a payment ... annually into the future forever' (the phrase borrowed from p.~1196 in \cite{Goldie+Maller:2000}). When studying the aforementioned random series from purely mathematical viewpoint, the one-sided assumptions are normally omitted whereas the term `perpetuity' is still used. See the books \cite{Buraczewski et al:2016} and \cite{Iksanov:2016} for surveys of the area of perpetuities from two different perspectives.

In the present paper we investigate the asymptotic behavior as $b\to 1-$ of the convergent series $\sum_{k\geq 0}b^{S_k}\eta_{k+1}$ that we call {\it discounted convergent perpetuity}.
We intend to prove the basic limit theorems for the discounted convergent perpetuities: a strong law of large numbers, a functional central limit theorem and a law of the iterated logarithm. Getting back to the actuarial interpretation, these results describe the fluctuations of the present value when the actuarial market is close to the customer-friendly scenario of no risk.

A sufficient condition for the almost sure (a.s.) absolute convergence of the random series $\sum_{k\geq 0}b^{S_k}\eta_{k+1}$ with fixed $b\in (0,1)$ is $\me\xi\in (0,\infty)$ and $\me \log^+|\eta|<\infty$, see, for instance, Theorem 2.1 in \cite{Goldie+Maller:2000}. This sufficient condition holds, that is, the discounted perpetuity is well-defined for all $b\in (0,1)$, under the assumptions of all our results to be formulated soon.

We start with a strong law of large numbers. 
\begin{thm}\label{main1}
Assume that $\mu:=\me\xi\in (0,\infty)$ and $\me |\eta|<\infty$. Then
\begin{equation}\label{aux51}
\lim_{b\to 1-}(1-b)\sum_{k\geq 0}b^{S_k}\eta_{k+1}=\mu^{-1}{\tt m}\quad \text{{\rm a.s.}},
\end{equation}
where ${\tt m}:=\me \eta$.
\end{thm}

Throughout the paper we write $\overset{\mmp}{\to}$ to denote convergence in probability, and $\Rightarrow$ and ${\overset{{\rm
d}}\longrightarrow}$ to denote weak convergence in a function space and weak convergence
of one-dimensional distributions, respectively. Also, we denote by $D(0,\infty)$ the Skorokhod space of right-continuous functions defined on
$(0,\infty)$ with finite limits from the left at positive points. We proceed by giving a functional central limit theorem.
\begin{thm}\label{main2}
Assume that $\mu=\me \xi\in (0,\infty)$, $\me \eta=0$ and ${\tt s}^2:={\rm Var}\,\eta\in (0,\infty)$. Then, as $b\to 1-$,
\begin{equation}\label{clt22}
\Big((1-b^2)^{1/2}\sum_{k\geq 0}b^{uS_k}\eta_{k+1}\Big)_{u>0} ~\Rightarrow~ (2{\tt s}^2\mu^{-1})^{1/2}\Bigg(\int_{[0,\,\infty)}e^{-uy}{\rm d}B(y)\Bigg)_{u>0}
\end{equation}
in the $J_1$-topology on $D(0,\infty)$, where $(B(t))_{t\geq 0}$ is a standard Brownian motion.
\end{thm}
\begin{rem}
The limit process in Theorem \ref{main2} is an a.s.\ continuous Gaussian process on $(0,\infty)$ with covariance
\begin{equation}\label{limit}
\me \int_{[0,\,\infty)}e^{-uy}{\rm d}B(y)\int_{[0,\,\infty)}e^{-vy}{\rm d}B(y)=\frac{1}{u+v},\quad u,v>0.
\end{equation}
Such a process has appeared in the recent articles \cite{Buraczewski+Dovgay+Iksanov:2020}, \cite{Iksanov+Kabluchko:2018} and \cite{Kabluchko:2019}. The latter paper provides additional references.
\end{rem}

Putting in \eqref{clt22} $u=1$ and using \eqref{limit} with $u=v=1$ we obtain a one-dimensional central limit theorem.
\begin{cor}\label{cormain2}
Under the assumptions of Theorem \ref{main2}, as $b\to 1-$,
$$(1-b^2)^{1/2}\sum_{k\geq 0}b^{S_k}\eta_{k+1}~{\overset{{\rm d}}\longrightarrow}~({\tt s}^2 \mu^{-1})^{1/2}\, {\rm Normal}(0,1),
$$ where ${\rm Normal}(0,1)$ denotes a random variable with the standard normal distribution.
\end{cor}

Finally, we are interested in the rate of a.s.\ convergence in Theorem \ref{main1} when ${\tt m}=0$ which is expressed by a law of the iterated logarithm. A hint concerning the form of this law is given by the central limit theorem, Corollary \ref{cormain2}.  For a family $(x_t)$ we denote by $C((x_t))$ the set of its limit points.
\begin{thm}\label{main}
Assume that $\mu=\me\xi\in (0,\infty)$, 
$\me \eta=0$ and ${\tt s}^2={\rm Var}\,\eta \in (0,\infty)$. Then
\begin{equation}\label{auxlimsup}
{\lim\sup\,(\lim
\inf)}_{b\to 1-}\Big(\frac{1-b^2}{\log\log \frac{1}{1-b^2}}\Big)^{1/2}\sum_{k\geq 0}b^{S_k}\eta_{k+1}=+(-)(2{\tt s}^2 \mu^{-1})^{1/2}\quad\text{{\rm a.s.}}
\end{equation}
In particular, $$C\bigg(\bigg(\Big(\frac{1-b^2}{2{\tt s}^2\mu^{-1}\log\log \frac{1}{1-b^2}}\Big)^{1/2}\sum_{k\geq 0}b^{S_k}\eta_{k+1}: b\in ((1-e^{-1})^{1/2},1)\bigg)\bigg)=[-1,1]\quad\text{{\rm a.s.}}$$
\end{thm}

\section{Related literature}

\noindent {\sc Random power series}. The random power (or geometric) series $\sum_{k\geq 0}b^k\eta_{k+1}$ for $b\in(0,1)$ is a rather particular case of a discounted convergent perpetuity which corresponds to the degenerate random walk $S_k=k$ for $k\in\mn_0$. In this section we first discuss known counterparts of our main results for the random power series.

\noindent {\it Law of large numbers}. Under the assumption $\me |\eta|<\infty$, the following strong law of large numbers can be found in Theorem 1 of \cite{Lai:1974}
\begin{equation}\label{aux152}
\lim_{b\to 1-}(1-b)\sum_{k\geq 0}b^k \eta_{k+1}={\tt m}\quad \text{a.s.},
\end{equation}
where ${\tt m}=\me \eta$.

\noindent {\it Central limit theorem}. Under the assumption $\me |\eta|^3<\infty$ Theorem 1 in \cite{Gerber:1971} proves a Berry-Ess\'{e}en inequality which entails
$$(1-b^2)^{1/2}\Big(\sum_{k\geq 0}b^k \eta_{k+1}-\frac{{\tt m}}{1-b}\Big)~{\overset{{\rm d}}\longrightarrow}~{\tt s}\, {\rm Normal}(0,1),\quad b\to 1-,$$ where ${\tt s}^2={\rm Var}\,\eta\in (0,\infty)$. Theorem 4.1 in \cite{Whitt:1972} is a functional limit theorem in the Skorokhod space for the process $(\sum_{k=0}^{\lfloor (1-b)^{-1}t \rfloor} b^k\eta_{k+1})_{t\geq 0}$, properly normalized and centered, as $b\to 1-$. Here and hereafter, $\lfloor x\rfloor$ denotes the integer part of real $x$. The corresponding limit process is a time-changed Brownian motion.

\noindent {\it Law of iterated logarithm}. It was proved in Theorem 3 of \cite{Gaposhkin:1965} that $${\lim\sup}_{b\to 1-}\Big(\frac{1-b^2}{\log\log\frac{1}{1-b^2}}\Big)^{1/2}\sum_{k\geq 0}b^k \eta_{k+1}=2^{1/2}{\tt s}$$ for centered bounded $\eta_k$ with variance ${\tt s}^2$. In Theorem 2 of \cite{Lai:1974} this limit relation was stated without proof, for not necessarily bounded $\eta_k$. Our Theorem \ref{main} is an analogue of Theorem 1.1 in \cite{Bovier+Picco:1993} dealing with the random power series. In Theorem 1.1 of \cite{Picco+Vares:1994} the sequence $(\eta_k)_{k\in\mn}$ is stationary, conditionally centered and ergodic with $\me \eta_1^2<\infty$. In this more general setting the authors prove a counterpart of \eqref{auxlimsup} for the corresponding random power series. Another proof in both settings based on a strong approximation result is given in Theorem 2.1 of \cite{Zhang:1997}. See also \cite{Fu+Huang:2016} and \cite{Stoica:2003} for related results.

Although the random power series is a toy example of perpetuities, transferring results from the former to the latter may be a challenge. To justify this claim, we only mention that while necessary and sufficient conditions for the a.s.\ convergence of random power series can be easily obtained (just use the Cauchy root test in combination with the Borel-Cantelli lemma), the corresponding result for perpetuities is highly non-trivial, see Theorem 2.1 in \cite{Goldie+Maller:2000} and its proof. The reason is clear: the random power series is a weighted sum of independent random variables, whereas it is not the case for perpetuities.

Investigation of (general) weighted sums of independent identically distributed random variables has been and still is a rather popular trend of research. We refrain from giving a survey and only mention recent contributions \cite{Aymone:2019, Aymone+Frometa+Misturini:2020+} in which a random Dirichlet series is analyzed.

\noindent {\sc Discounted perpetuities}. As far as we know, Theorems \ref{main1}, \ref{main2} and \ref{main} are new. Under the additional assumption $\me\xi^2<\infty$ (we only require $\me\xi\in(0,\infty)$) our Corollary \ref{cormain2} follows from Theorem 6.1 in \cite{Vervaat:1979} which we state as Proposition \ref{verv} for reader's convenience.
\begin{assertion}\label{verv}
Assume that $\mu=\me\xi \in (0, \infty)$, $\sigma^2={\rm Var}\,\xi\in [0,\infty)$, ${\tt s}^2={\rm Var}\, \eta\in [0,\infty)$, $\sigma^2+{\tt s}^2>0$. Then $$\alpha^{-1/2}\Big(\sum_{k\geq 1}e^{-\alpha S_{k-1}}\eta_k-\alpha {\tt m}\mu^{-1}\Big)~{\overset{{\rm d}}\longrightarrow}~v\, {\rm Normal}(0,1),\quad \alpha\to\infty,$$ where ${\rm Normal}(0,1)$ denotes a random variable with the standard normal distribution, ${\tt m}=\me \eta$, $v^2:=2^{-1}\sigma^2\mu^{-3}{\tt m}^2+\gamma {\tt m}\mu^{-2}+ 2^{-1}\sigma^2\mu^{-1}$ and $\gamma:=\me \xi \eta-\mu{\tt m}\in\mr$.
\end{assertion}
We stress that our idea of proof of Theorem \ref{main2} is different from Vervaat's. Also, we note that in Theorem 2 of \cite{DallAglio:1964} the method of moments is employed for proving a (one-dimensional) central limit theorem for $\sum_{k\geq 0}b^{S_k}$ as $b\to 1-$ under the assumptions $\xi\geq 0$ a.s.\ and $\me \xi^p<\infty$ for all $p>0$.

\section{Proof of Theorem \ref{main1}}

We shall use a fragment of Theorem 5 on p.~49 in \cite{Hardy:1949} that we give in a form adapted to our setting.
\begin{lemma}\label{har}
Let $(c_k(b))_{k\in\mn}$ and $(s_k)_{k\in\mn}$ be sequences of real-valued functions defined on $(0,1)$ and real numbers, respectively. Assume that

\noindent (i) $\sum_{k\geq 1}|c_k(b)|<\infty$ for all $b\in (0,1)$ and that, for some $b_0\in (0,1)$ and some $A>0$ which does not depend on $b$, $\sum_{k\geq 1}|c_k(b)|\leq A$ for all $b\in (b_0,1)$;

\noindent (ii) $\lim_{b\to 1-} c_k(b)=0$ for all $k\in\mn$;

\noindent (iii) $\lim_{b\to 1-}\sum_{k\geq 1}c_k(b)=1$.

Then $t(b):=\sum_{k\geq 1}c_k(b)s_k$ converges for all $b\in (0,1)$. Furthermore, if $\lim_{n\to\infty}s_n=s\in\mr$, then $\lim_{b\to 1-}t(b)=s$.

\end{lemma}

\begin{proof}[Proof of Theorem \ref{main1}]
We first prove that
\begin{equation}\label{aux50}
\lim_{b\to 1-}(1-b)\sum_{k\geq 0}b^{S_k}=\mu^{-1}\quad\text{a.s.}
\end{equation}
For $x\in\mr$, put $M(x)=\#\{n\geq 0: S_n\leq x\}$. Since $\lim_{n\to\infty}S_n=+\infty$ a.s., we have $M(x)<\infty$ a.s. Furthermore, by Theorem B in \cite{Lai:1975}, $\lim_{x\to\infty}x^{-1}M(x)=\mu^{-1}$ a.s. Hence, given $\varepsilon>0$ there exists an a.s.\ finite $x_0>0$ such that $|x^{-1}M(x)-\mu^{-1}|\leq \varepsilon$ whenever $x\geq x_0$. Write
$$\sum_{k\geq 0}b^{S_k}=\sum_{k\geq 0}b^{S_k}\1_{\{S_k\leq x_0\}}+\int_{(x_0,\,\infty)} b^x{\rm d}M(x).$$ The number of summands in the sum on the right-hand side is a.s.\ finite, for it is equal to $M(x_0)$, whence $\lim_{b\to 1-}\sum_{k\geq 0}b^{S_k}\1_{\{S_k\leq x_0\}}=M(x_0)$ a.s. Integration by parts yields $$\int_{(x_0,\,\infty)} b^x{\rm d}M(x)+b^{x_0}M(x_0)=|\log b| \int_{x_0}^\infty b^x M(x){\rm d}x\leq (\mu^{-1}+\varepsilon)b^{x_0}(1+|\log b|x_0)/|\log b|.$$ Thus, $${\lim\sup}_{b\to 1-}(1-b)\sum_{k\geq 0}b^{S_k}\leq \mu^{-1}\quad\text{a.s.}$$ The proof of the converse inequality for the limit inferior is completely analogous.

Passing to the proof of \eqref{aux51} we use summation by parts to obtain, for $b\in (0,1)$ and $\ell\in\mn$,
\begin{equation}\label{summat}
\sum_{k=1}^\ell b^{S_{k-1}}\eta_k=\sum_{k=1}^{\ell-1}(b^{S_{k-1}}-b^{S_k})T_k +b^{S_{\ell-1}}T_\ell,
\end{equation}
where $T_0:=0$ and $T_k:=\eta_1+\ldots+\eta_k$ for $k\in\mn$. We have $\lim_{\ell\to\infty}b^{S_{\ell-1}}T_\ell=0$ a.s.\ because by the strong law of large numbers the first factor decreases to zero exponentially fast, whereas the second factor exhibits at most  linear growth. Hence, $$\sum_{k\geq 1}b^{S_{k-1}}\eta_k=\sum_{k\geq 1}k(b^{S_{k-1}}-b^{S_k})(k^{-1}T_k).$$ We are going to apply Lemma \ref{har} with $c_k(b):=\mu(1-b)k(b^{S_{k-1}}-b^{S_k})$ for $k\in\mn$ and $b\in(0,1)$ and $s_k:=k^{-1}T_k$ for $k\in\mn$. While (ii) of Lemma \ref{har} holds trivially (a.s.), (iii) is a consequence of $\sum_{k\geq 1}c_k(b)=\mu (1-b)\sum_{k\geq 0}b^{S_k}$ and \eqref{aux50}. Let us prove (i). By another appeal to the strong law of large numbers, given $\varepsilon\in (0,\mu)$, there exists a random integer $N$ such that $b^{S_{k-1}}\leq b^{(\mu-\varepsilon)(k-1)}$ whenever $k\geq N+1$. Fix any $b_1\in (0,1)$. By the mean value theorem for differentiable functions, for $k\geq N+1$ and $b\in (b_1, 1)$,
\begin{equation}\label{aux53}
|b^{S_{k-1}}-b^{S_k}|\leq \max (b^{S_{k-1}}, b^{S_k})|\log b||\xi_k|\leq b_1^{-(\mu-\varepsilon)}b^{(\mu-\varepsilon)k}|\log b||\xi_k|.
\end{equation}
Using the inequality $xe^{-x}\leq 2e^{-x/2}$ for $x\geq 0$ we infer, for $k\geq N+1$ and $b\in (b_1, 1)$,
$$k|b^{S_{k-1}}-b^{S_k}|\leq 2(\mu-\varepsilon)^{-1}b_1^{-(\mu-\varepsilon)}b^{(\mu-\varepsilon)k/2}|\xi_k|=:cb^{(\mu-\varepsilon)k/2}|\xi_k|.$$ With this at hand, for $b\in (b_1,1)$,
\begin{multline*}
(\mu(1-b))^{-1}\sum_{k\geq 1}|c_k(b)|=\sum_{k\geq 1}k|b^{S_{k-1}}-b^{S_k}|\leq 2\sum_{k=1}^N k+\sum_{k\geq N+1}k|b^{S_{k-1}}-b^{S_k}|\leq N(N+1)\\+c\sum_{k\geq 1}b^{(\mu-\varepsilon)k/2}|\xi_k|. \end{multline*}
In view of \eqref{aux152}, $\lim_{b\to 1-}(1-b)\sum_{k\geq 1}b^{(\mu-\varepsilon)k/2}|\xi_k|=2|\me \xi|/(\mu-\varepsilon)$ a.s. This justifies (i) in the present setting.

By the strong law of large numbers $\lim_{k\to\infty}s_k=\lim_{k\to\infty}(k^{-1}T_k)={\tt m}$ a.s. Invoking Lemma \ref{har} we arrive at \eqref{aux51}.
The proof of Theorem \ref{main1} is complete.
\end{proof}

Later on, we shall need the following result. Its proof is omitted, for it is analogous to the proof of Theorem \ref{main1}.
\begin{lemma}\label{stronglaw}
Assume that $\me |\eta|<\infty$. Let $(x_n)_{n\in\mn}$ and $(y_n)_{n\in\mn}$ be sequences of numbers in $(0,1)$ approaching $1$ as $n\to\infty$. Let $\lambda>0$ and $M: (0,1)\to \mn$ be a function satisfying $\lim_{n\to\infty} M(x_n)(1-y_n^\lambda)=a\in [0,\infty]$. If $a=\infty$, then
$$\lim_{n\to\infty}\frac{1-y_n^\lambda}{y_n^{\lambda M(x_n)}} \sum_{k\geq M(x_n)+1}y_n^{\lambda k}\eta_k 
={\tt m}\quad\text{{\rm a.s.}},$$ where ${\tt m}=\me \eta$; if $a\in [0,\infty)$, then
$$\lim_{n\to\infty}
(1-y_n^\lambda)\sum_{k\geq M(x_n)+1}y_n^{\lambda k}\eta_k={\tt m}e^{-a}\quad\text{{\rm a.s.}}$$ 

Clearly, these limit relations also hold if we put formally $x_n=y_n=b$ and let $b\to 1-$, that is, if one passes to the limit continuously.
\end{lemma}

\section{Proof of Theorem \ref{main2}}

We shall prove weak convergence of the finite-dimensional distributions and then tightness.

\subsection{Proof of the finite-dimensional distributions in \eqref{clt22}}

We shall use the Cram\'{e}r-Wold device. Namely, we intend to show that, for
any $\ell\in\mn$, any real $\alpha_1,\ldots, \alpha_\ell$ and any
$0<u_1<\ldots<u_\ell<\infty$, as $b\to 1-$,
\begin{equation}\label{fidi}
(1-b^2)^{1/2}\sum_{i=1}^\ell \alpha_i \sum_{k\geq 0}b^{u_iS_k}\eta_{k+1} ~{\overset{{\rm d}}\longrightarrow}~(2{\tt s}^2\mu^{-1})^{1/2}\sum_{i=1}^\ell\alpha_i \int_{[0,\,\infty)}e^{-u_i y}{\rm d}B(y).
\end{equation}

For $k\in\mn$, denote by $\mathcal{F}_k$ the $\sigma$-algebra generated by $(\xi_j, \eta_j)_{1\leq j\leq k}$. We shall write $\me_k(\cdot)$ for $\me(\cdot|\mathcal{F}_k)$. For each $b\in (0,1)$, the sequence $$\Big((1-b^2)^{1/2}\sum_{i=1}^\ell \alpha_i \sum_{k=0}^{n-1} b^{u_iS_k}\eta_{k+1}, \mathcal{F}_n\Big)_{n\in\mn}$$ forms a martingale (the martingale is not necessarily integrable, for the situation that $\me b^\xi=\infty$ is not excluded). By the martingale central limit theorem (Theorem 2.5(a) in \cite{Helland:1982}), \eqref{fidi} follows if we can show that
\begin{equation}\label{cova}
(1-b^2) \sum_{k\geq 0}\me_k \Big(\sum_{i=1}^\ell \alpha_i b^{u_iS_k}\eta_{k+1}\Big)^2 ~\overset{\mmp}{\to}~ 
2{\tt s}^2\mu^{-1}\me \Big(\sum_{i=1}^\ell \alpha_i\int_{[0,\,\infty)}e^{-u_iy}{\rm d}B(y)\Big)^2,\quad b\to 1-
\end{equation}
and, for all $\varepsilon>0$,
\begin{equation}\label{linde}
(1-b^2) \sum_{k\geq 0}\me_k \Big(\sum_{i=1}^\ell \alpha_i b^{u_iS_k}\eta_{k+1}\Big)^2\1_{\{(1-b^2)^{1/2}|\sum_{i=1}^\ell \alpha_i b^{u_iS_k}\eta_{k+1}|>\varepsilon\}}~\overset{\mmp}{\to}~0,\quad b\to 1-.
\end{equation}

We start by proving \eqref{cova}:
$$(1-b^2) \sum_{k\geq 0}\me_k \Big(\sum_{i=1}^\ell \alpha_i b^{u_iS_k}\eta_{k+1}\Big)^2={\tt s}^2 (1-b^2)\Big(\sum_{i=1}^\ell\alpha_i^2\sum_{k\geq 0} b^{2u_iS_k}+2\sum_{1\leq i<j\leq \ell}\alpha_i\alpha_j\sum_{k\geq 0} b^{(u_i+u_j)S_k}\Big).$$ By Theorem \ref{main1}, this converges a.s., as $b\to 1-$, to $${\tt s}^2 \mu^{-1}\Big(\sum_{i=1}^\ell \alpha_i^2 u_i^{-1}+4\sum_{1\leq i<j\leq \ell}\alpha_i\alpha_j (u_i+u_j)^{-1}\Big)=2{\tt s}^2 \mu^{-1}\me \Big(\sum_{i=1}^\ell \alpha_i\int_{[0,\,\infty)}e^{-u_iy}{\rm d}B(y)\Big)^2,$$ where the last equality follows from \eqref{limit}.

Passing to the proof of \eqref{linde} we first conclude that, in view of
\begin{multline*}
(a_1+\ldots+a_\ell)^2\1_{\{|a_1+\ldots++a_\ell|>y\}}\leq (|a_1|+\ldots+|a_\ell|)^2\1_{\{|a_1|+\ldots+|a_\ell|>y\}}\\\leq
\ell^2(|a_1|\vee\ldots\vee |a_\ell|)^2\1_{\{\ell(|a_1|\vee\ldots\vee |a_\ell|)>y\}}\leq \ell^2(a_1^2\1_{\{|a_1|>y/\ell\}}+\ldots+a_\ell^2\1_{\{|a_\ell|>y/\ell\}})
\end{multline*}
which holds for $a_1,\ldots, a_\ell\in \mr$ and $y>0$, it suffices to show that, for all $\varepsilon>0$ and $u>0$,
$$(1-b^2) \sum_{k\geq 0}\me_k (b^{uS_k}\eta_{k+1})^2\1_{\{(1-b^2)^{1/2}b^{uS_k}|\eta_{k+1}|>\varepsilon\}}~\overset{\mmp}{\to}~0,\quad b\to 1-.$$ Put $T:=\sup\{n\in\mn_0: S_n\leq 0\}$ and note that $T<\infty$ a.s.\ as a consequence of $\lim_{n\to\infty}S_n=+\infty$ a.s. We infer 
$$(1-b^2)\sum_{k=0}^T \me_k (b^{uS_k}\eta_{k+1})^2\1_{\{(1-b^2)^{1/2}b^{uS_k}|\eta_{k+1}|>\varepsilon\}}\leq {\tt s}^2 (1-b^2)\sum_{k=0}^T b^{2uS_k}~\to~0\quad\text{a.s.~as}~~b\to 1-.$$ To proceed, observe that, for $k\geq T+1$, we have $b^{uS_k}\leq 1$, whence $$\{(1-b^2)^{1/2}b^{uS_k}|\eta_{k+1}|>\varepsilon\}\subseteq\{|\eta_{k+1}|>\varepsilon (1-b^2)^{-1/2}\}.$$ This yields
\begin{multline*}
(1-b^2)\sum_{k\geq T+1} \me_k (b^{uS_k}\eta_{k+1})^2\1_{\{(1-b^2)^{1/2}b^{uS_k}|\eta_{k+1}|>\varepsilon\}}\\\leq \me \eta^2\1_{\{|\eta|>\varepsilon (1-b^2)^{-1/2}\}} (1-b^2)\sum_{k\geq 0} b^{2uS_k}~\to~0\quad\text{a.s.~as}~~b\to 1-.
\end{multline*}
The limit relation is justified by the fact that while the truncated second moment converges to $0$, $\lim_{b\to 1-}(1-b^2)\sum_{k\geq 0} b^{2uS_k}=(\mu u)^{-1}$ a.s.\ by Theorem \ref{main1}.

For the proof of Proposition \ref{lilhalf2} we need the following one-dimensional central limit theorem.
\begin{lemma}\label{lem:aux}
Let $M: (0,1)\to \mn$ satisfy $\lim_{b\to 1-}M(b)=\infty$. Under the assumptions of Theorem \ref{main2}, as $b\to 1-$,
$$\Big(\sum_{k=0}^{M(b)}b^{2\mu k}\Big)^{-1/2} \sum_{k=0}^{M(b)} b^{S_k}\eta_{k+1}~{\overset{{\rm d}}\longrightarrow}~({\tt s}^2 \mu^{-1})^{1/2}\, {\rm Normal}(0,1).
$$ 
\end{lemma}

After noting that $\sum_{k=0}^{M(b)}b^{2S_k}\sim \sum_{k=0}^{M(b)}b^{2\mu k}$ a.s.\ as $b\to 1-$ by the strong law of large numbers for random walks, a simplified version of the proof given above applies. We omit details.

\subsection{Proof of tightness in \eqref{clt22}}

Fix any $c,d\in (0,\infty)$, $c<d$. We have to prove tightness on $[c,d]$.

For each $\delta\in (0,\mu)$ and $k\in\mn_0$, define the event $\mathcal{R}_k(\delta):=\{|S_k-\mu k|>\delta k\}$. We first check that
\begin{equation}\label{aux67}
\lim_{b\to 1-}(1-b^2)^{1/2}\sup_{u\in [c,\,d]}\, \Big|\sum_{k\geq 0}b^{uS_k}\1_{\mathcal{R}_k(\delta)}\eta_{k+1}\Big|=0\quad\text{a.s.}
\end{equation}
Indeed, the supremum does not exceed a.s. $$\sum_{k\geq 0}b^{cS_k}\1_{\{S_k\geq 0\}}\1_{\mathcal{R}_k(\delta)}|\eta_{k+1}|+\sum_{k\geq 0}b^{dS_k}\1_{\{S_k<0\}}\1_{\mathcal{R}_k(\delta)}|\eta_{k+1}|.$$ Here, each summand converges a.s.\ as $b\to 1-$ to an a.s.\ finite random variable. Furthermore, the number of nonzero summands is a.s.\ finite in view of $\sum_{k\geq 0}\1_{\mathcal{R}_k(\delta)}<\infty$ a.s.\ which is a consequence of the strong law of large numbers. Thus, \eqref{aux67} has been proved.

Next, we intend to show that, for any $u,v\in [c,d]$ and $b<1$ close to $1$,
\begin{equation}\label{aux68}
(1-b^2)\me \Big(\sum_{k\geq 0}(b^{uS_k}-b^{vS_k})\1_{\mathcal{R}^c_k(\delta)}\eta_{k+1}\Big)^2\leq A(u-v)^2
\end{equation}
for a constant $A$ which does not depend on $u$ and $v$. Here, $\mathcal{R}^c_k(\delta)$ denotes the complement of $\mathcal{R}_k(\delta)$, that is, $\mathcal{R}^c_k(\delta)=\{|S_k-\mu k|\leq \delta k\}$. To this end, we observe that $\mathcal{R}^c_k(\delta)\subseteq \{S_k>0\}$ and then invoking the mean value theorem for differentiable functions we obtain a.s.\ on $\mathcal{R}^c_k(\delta)$
\begin{multline*}
|b^{u S_k}-b^{v S_k}|\leq b^{cS_k}|\log b||u-v| S_k \leq  (\mu+\delta) b^{c(\mu-\delta)k}|\log b||u-v|k\\ \leq 2(\mu+\delta)(ce(\mu-\delta))^{-1} b^{(c/2)(\mu-\delta)k}|u-v|.
\end{multline*}
We have used the inequality $$\sup_{x>0}|\log b|xb^x\leq 1/e$$ for the last step. It remains to note that
\begin{multline*}
\me \Big(\sum_{k\geq 0}(b^{uS_k}-b^{vS_k})\1_{\mathcal{R}^c_k(\delta)}\eta_{k+1}\Big)^2={\tt s}^2\me \sum_{k\geq 0}(b^{uS_k}-b^{vS_k})^2 \1_{\mathcal{R}^c_k(\delta)}\\ \leq
4{\tt s}^2(\mu+\delta)^2(ce(\mu-\delta))^{-2} \sum_{k\geq 0}b^{c(\mu-\delta)k} (u-v)^2
\end{multline*}
and that 
$$\lim_{b\to 1-}(1-b^2)\sum_{k\geq 0}b^{c(\mu-\delta)k}=2(c(\mu-\delta))^{-1}.$$ Thus, \eqref{aux68} holds with $A=16 {\tt s}^2(\mu+\delta)^2 e^{-2}c^{-3}(\mu-\delta)^{-3}$. By formula (12.51) on p.~95 in \cite{Billingsley:1968}, the distributions of $$\Big((1-b^2)^{1/2}\sum_{k\geq 0}b^{uS_k}\1_{\mathcal{R}^c_k(\delta)}\eta_{k+1}\Big)_{u\in [c,\,d]}$$ are tight. The proof of Theorem \ref{main2} is complete.

\section{Proof of Theorem \ref{main}}

Our argument follows closely the paths of (slightly different) proofs of Theorem 1.1 in \cite{Bovier+Picco:1993} and Theorem 1.1 in \cite{Picco+Vares:1994}. In the cited references $S_n=n$, $n\in\mn_0$, that is, the random walk $(S_n)_{n\in\mn_0}$ is deterministic. Of course, we know that in our setting, for large $n$, $S_n$ is approximately $\mu n$ by the strong law of large numbers. Thus, an additional effort is needed to justify the replacement of $S_n$ with $\mu n$.

We start by proving an intermediate result.
\begin{assertion}\label{lilhalf}
Under the assumptions of Theorem \ref{main},
\begin{equation}\label{lil1}
{\lim\sup}_{b\to 1-}\Big(\frac{1-b^2}{\log\log \frac{1}{1-b^2}}\Big)^{1/2}\sum_{k\geq 0}b^{S_k}\eta_{k+1}\leq (2{\tt s}^2 \mu^{-1})^{1/2}\quad\text{{\rm a.s.}}
\end{equation}
and
\begin{equation}\label{lil2}
{\lim\inf}_{b\to 1-}\Big(\frac{1-b^2}{\log\log \frac{1}{1-b^2}}\Big)^{1/2}\sum_{k\geq 0}b^{S_k}\eta_{k+1}\geq -(2{\tt s}^2 \mu^{-1})^{1/2}\quad\text{{\rm a.s.}}
\end{equation}
\end{assertion}

We can and do assume that $\mu={\tt s}^2=1$. To see this, replace $b^{S_{k-1}}\eta_k$ with $b^{S_{k-1}/\mu}\eta_k/{\tt s}$ and note that $1-b^{2\mu}\sim \mu (1-b^2)$ as $b\to 1-$. Pick any $\delta\in (0,1)$. For $b\in (0,1)$ and such a $\delta$, put $$N_{2,\,\delta}(b):=\Big\lfloor \frac{1}{1-b^{2\delta}}\log\frac{1}{1-b^{2\delta}}\Big\rfloor$$ and, for $b\in [(1-e^{-1})^{1/2}, 1)$, put $$f(b):=\Big(2\frac{1}{1-b^2}\log\log\frac{1}{1-b^2}\Big)^{-1/2}.$$ We prove Proposition \ref{lilhalf} via a sequence of lemmas.
\begin{lemma}\label{1}
$\lim_{b\to 1-} f(b)\sum_{k\geq N_{2,\,\delta}(b)}b^{S_{k-1}}\eta_k=0$ {\rm a.s.}
\end{lemma}
\begin{proof}
Pick any increasing sequence $(b_n)_{n\in\mn}$ of positive numbers satisfying $\lim_{n\to\infty}b_n=1$,
\begin{equation}\label{aux0}
b_{n+1}-b_n~\sim~c_1(1-b_n)^{1+c_2},\quad n\to\infty
\end{equation}
for some $c_1, c_2>0$ and
\begin{equation}\label{aux1}
\sum_{n\geq n_0}(1-b_n)<\infty
\end{equation}
for some $n_0\in\mn$. One particular sequence satisfying these assumptions is given by $b_n=1-n^{-2}$ for $n\in\mn$ (with $c_1=2$ and $c_2=1/2$ in \eqref{aux0}). Note that \eqref{aux0} entails
\begin{equation*}
\lim_{n\to\infty}(1-b_{n+1})/(1-b_n)=1.
\end{equation*}

Suppose we can prove that, for all $\varepsilon>0$, 
$$I:=\sum_{n\geq n_0}\mmp\Big\{\sup_{b\in [b_n,\, b_{n+1}]}\Big|\sum_{k\geq N_{2,\,\delta}(b)}b^{S_{k-1}}\eta_k\Big|>\varepsilon/ f(b_n)\Big\}<\infty.$$ Then, by the Borel–Cantelli lemma, $$\sup_{b\in [b_n,\, b_{n+1}]}\Big|\sum_{k\geq N_{2,\,\delta}(b)}b^{S_{k-1}}\eta_k\Big|\leq\varepsilon/ f(b_n)$$ for $n$ large enough a.s. Since $f$ is nonnegative and decreasing  on $[(1-e^{-1})^{1/2}, 1)$, we have, for all large enough $n$, $$\Big|\sum_{k\geq N_{2,\,\delta}(b)}b^{S_{k-1}}\eta_k\Big|\leq \sup_{b\in [b_n,\, b_{n+1}]}\Big|\sum_{k\geq N_{2,\,\delta}(b)}b^{S_{k-1}}\eta_k\Big|\leq \varepsilon/f(b_{n})\leq \varepsilon/f(b)$$ a.s.\ whenever $b\in [b_n, b_{n+1}]$. Hence, ${\lim\sup}_{b\to 1-}f(b)\sum_{k\geq N_{2,\,\delta}(b)}b^{S_{k-1}}\eta_k\leq \varepsilon$ a.s.\ which entails the claim.

Since the function $N_{2,\,\delta}$ is nondecreasing on $(0,1)$ we obtain
$$\sup_{b\in [b_n,\,b_{n+1}]}\Big|\sum_{k\geq N_{2,\,\delta}(b)}b^{S_{k-1}}\eta_k\Big|\leq \sup_{b\in [b_n,\,b_{n+1}]}\sum_{k\geq N_{2,\,\delta}(b_n)}b^{S_{k-1}}|\eta_k|.$$ Further, by the strong law of large numbers, for large $n$, 
the latter is estimated from above by $$\sup_{b\in [b_n,\,b_{n+1}]}\sum_{k\geq N_{2,\,\delta}(b_n)}b^{\delta (k-1)}|\eta_k|\leq 
\sum_{k\geq N_{2,\,\delta}(b_n)}b_{n+1}^{\delta(k-1)}\me |\eta_k|+\sum_{k\geq N_{2,\,\delta}(b_n)}b_{n+1}^{\delta(k-1)}(|\eta_k|-\me |\eta_k|).$$ 
Thus, noting that $\me |\eta_k|\leq 1$,
\begin{eqnarray}\label{aux6}
I\leq \sum_{n\geq n_0}\1_{\{\sum_{k\geq N_{2,\,\delta}(b_n)}b_{n+1}^{\delta (k-1)}>\varepsilon/(2f(b_n))\}}+ \sum_{n\geq n_0}\mmp\Big\{\sum_{k\geq N_{2,\,\delta}(b_n)}b_{n+1}^{\delta(k-1)}(|\eta_k|-\me |\eta_k|)>\varepsilon/(2f(b_n))\Big\}. 
\end{eqnarray}
Using \eqref{aux0} and $-\log x=(1-x)+O((1-x)^2)$ as $x\to 1-$ we obtain
\begin{multline*}
f(b_n) \sum_{k\geq N_{2,\,\delta}(b_n)}b_{n+1}^{\delta(k-1)}= f(b_n)\frac{b_{n+1}^{\delta(N_{2,\,\delta}(b_n)-1)}}{1-b_{n+1}^\delta}~\sim~ f(b_n) \frac{(1-b^{2\delta}_n)^{1/2}}{1-b_n^\delta}\\\sim \frac{(1-b^2_n)^{1/2}}{(2\log\log (1/(1-b_n^2)))^{1/2}}\frac{2\delta^{-1/2}}{(1-b^2_n)^{1/2}}~\to~ 0,\quad n\to\infty.
\end{multline*}
This proves that the first series on the right-hand side of \eqref{aux6} trivially converges, for it contains finitely many nonzero summands. By Markov's inequality and \eqref{aux0}, the probability in the second series is upper bounded by
\begin{multline*}
4\varepsilon^{-2}f^2(b_n) \sum_{k\geq N_{2,\,\delta}(b_n)}b_{n+1}^{2\delta(k-1)}=4\varepsilon^{-2}f^2(b_n)\frac{b_{n+1}^{2\delta(N_{2,\,\delta}(b_n)-1)}}{1-b_{n+1}^{2\delta}}~
\sim~2\varepsilon^{-2}\frac{1-b^2_n}{\log\log (1/(1-b_n^2))} 
,\quad n\to\infty.
\end{multline*}
In view of \eqref{aux1}, this is the general term of a convergent series. Hence, the second series on the right-hand side of \eqref{aux6} converges. The proof of Lemma \ref{1} 
is complete.
\end{proof}

For $b\in (0,1)$ close to $1$, $\delta$ as above and $\theta>0$, put $$N_{1,\,\delta,\,\theta}(b):=\Big\lfloor \frac{1+\theta}{1-b^{2\delta}}\log\log\frac{1}{1-b^{2\delta}}\Big\rfloor.$$
\begin{lemma}\label{2}
$\lim_{b\to 1-}f(b)\sum_{k=N_{1,\,\delta,\,\theta}(b)+1}^{N_{2,\,\delta}(b)}b^{S_{k-1}}\eta_k=0$ {\rm a.s.}
\end{lemma}
\begin{proof}
Similarly to \eqref{summat}, summation by parts yields $$\sum_{k=N_{1,\,\delta,\,\theta}(b)+1}^{N_{2,\,\delta}(b)}b^{S_{k-1}}\eta_k=\sum_{k=N_{1,\,\delta,\,\theta}(b)+1}^{N_{2,\,\delta}(b)-1}(b^{S_{k-1}}-b^{S_k})T_k+
b^{S_{N_{2,\,\delta}(b)-1}}T_{N_{2,\,\delta}(b)}-b^{S_{N_{1,\,\delta,\,\theta}(b)}}T_{N_{1,\,\delta,\,\theta}(b)},$$ where, as in the proof of Theorem \ref{main1}, $T_k=\eta_1+\ldots+\eta_k$ for $k\in\mn$. By the strong law of large numbers, for $b$ close to $1$,
$$|b^{S_{N_{2,\,\delta}(b)-1}}T_{N_{2,\,\delta}(b)}-b^{S_{N_{1,\,\delta,\,\theta}(b)}}T_{N_{1,\,\delta,\,\theta}(b)}|\leq b^{\delta (N_{2,\,\delta}(b)-1)}|T_{N_{2,\,\delta}(b)}|+ b^{\delta N_{1,\,\delta,\,\theta}(b)}|T_{N_{1,\,\delta,\,\theta}(b)}|.$$ One can check that
\begin{equation}\label{aux55}
b^{\delta N_{1,\,\delta,\,\theta}(b)}~\sim~\Big(\log\frac{1}{1-b^{2\delta}}\Big)^{-(1+\theta)/2}\quad\text{and}\quad b^{\delta N_{2,\,\delta}(b)}~\sim~(1-b^{2\delta})^{1/2}\quad\text{a.s.~as}~~b\to 1-.
\end{equation}
Further, recall that, as $\ell\to\infty$, 
\begin{equation}\label{lilweak}
|T_\ell|\leq \sup_{k\leq \ell}\,|T_k|=O\big((\ell \log\log \ell)^{1/2}\big)\quad \text{a.s.}
\end{equation}
by the law of the iterated logarithm for standard random walks. Using this limit relation we infer 
\begin{multline}\label{aux120}
f(b)|b^{S_{N_{2,\,\delta}(b)-1}}T_{N_{2,\,\delta}(b)}-b^{S_{N_{1,\,\delta,\,\theta}(b)}}T_{N_{1,\,\delta,\,\theta}(b)}|=O(((1-b)\log (1/(1-b)))^{1/2}) 
\\+O\Big(\Big(\frac{\log\log (1/(1-b))}{(\log (1/(1-b)))^{1+\theta}}\Big)^{1/2}\Big)~\to~0\quad\text{a.s.~as}~~b\to 1-.
\end{multline}
According to \eqref{aux53}, for $b$ close to $1$, $$\Big|\sum_{k=N_{1,\,\delta,\,\theta}(b)+1}^{N_{2,\,\delta}(b)-1}(b^{S_{k-1}}-b^{S_k})T_k\Big|\leq {\rm const}\,|\log b|(\sup_{k\leq N_{2,\,\delta}(b)}|T_k|)\sum_{k\geq N_{1,\,\delta,\,\theta}(b)+1} b^{\delta k} |\xi_k|.$$  With the help of \eqref{aux55} we obtain $$\sum_{k\geq N_{1,\,\delta,\,\theta}(b)+1} b^{\delta k} |\xi_k|~\sim~\me |\xi|\frac{b^{\delta N_{1,\,\delta,\,\theta}(b)}}{1-b^\delta}~\sim~\frac{\me |\xi|}{(1-b^\delta)(\log(1/(1-b^{2\delta})))^{(1+\theta)/2}}\quad\text{a.s.~as}~~b\to 1-$$ by an application of Lemma \ref{stronglaw} with $\eta=|\xi|$ and $a=\infty$. This in combination with \eqref{lilweak} yields
\begin{equation}\label{eq:aux4}
f(b)\Big|\sum_{k=N_{1,\,\delta,\,\theta}(b)+1}^{N_{2,\,\delta}(b)-1}(b^{S_{k-1}}-b^{S_k})T_k\Big|=O\Big(\frac{1}{(\log (1/(1-b)))^{\theta/2}}\Big)~\to~0\quad\text{a.s.~as}~~b\to 1-.
\end{equation}
The proof of Lemma \ref{2} is complete.
\end{proof}

For $b\in (0,1)$ close to $1$, put $$N_2(b):=\Big\lfloor \frac{1}{1-b^2}\log\frac{1}{1-b^2}\Big\rfloor.$$ We claim that $$\lim_{b\to 1-}f(b)\sum_{k=N_2(b)+1}^{N_{2,\,\delta}(b)}b^{S_{k-1}}\eta_k=0\quad {\rm a.s.}$$ For the most part, this follows by repeating the proof of Lemma \ref{2} with $N_2(b)$ replacing $N_{1,\,\delta,\,\theta}(b)$, the only changes being that the second summand on the right-hand side of \eqref{aux120} and the right-hand side of \eqref{eq:aux4} are $O(((1-b)^\delta \log (1/(1-b)))^{1/2})$ as $b\to 1-$. The last centered formula in combination with Lemma \ref{1} enable us to conclude that
\begin{equation}\label{aux121}
\lim_{b\to 1-}f(b)\sum_{k\geq N_2(b)+1} b^{S_{k-1}}\eta_k=0\quad {\rm a.s.}
\end{equation}
This limit relation will be used in the proof of Proposition \ref{lilhalf2}.

Denote by $\mathcal{F}_0$ the trivial $\sigma$-algebra and recall that, for $k\in\mn$, $\mathcal{F}_k$ denotes the $\sigma$-algebra generated by $(\xi_j, \eta_j)_{1\leq j\leq k}$ and that, for $k\in\mn_0$, we write $\me_k(\cdot)$ for $\me(\cdot|\mathcal{F}_k)$.
\begin{lemma}\label{3}
For all $\rho>0$,
\begin{equation}\label{aux010}
\lim_{b\to 1-}f(b)\sum_{k=1}^{N_{1,\,\delta,\,\theta}(b)}b^{S_{k-1}}\eta_k\1_{\mathcal{S}_k(b)}
=0\quad\text{{\rm a.s.}}
\end{equation}
and
\begin{equation}\label{aux011}
\lim_{b\to 1-}f(b)\sum_{k=1}^{N_{1,\,\delta,\,\theta}(b)}b^{S_{k-1}}\me_{k-1}(\eta_k\1_{\mathcal{S}_k(b)})
=0\quad\text{{\rm a.s.}},
\end{equation}
where $\mathcal{S}_k(b):=\{|\eta_k|>\rho b^{-S_{k-1}}((1-b^{2\delta})\log\log (1/(1-b^{2\delta})))^{-1/2}\}$.
\end{lemma}
\begin{proof}
We only give a detailed proof of \eqref{aux010} and then explain which modifications are needed for a proof of \eqref{aux011}.

\noindent {\sc Proof of \eqref{aux010}}. For $b\in (0,1)$ and the same $\delta\in (0,1)$ as before, put $$N_\delta(b):=\Big\lfloor\frac{1}{1-b^{2\delta}}\Big\rfloor.$$ Plainly, for all $\rho>0$,
$$\lim_{b\to 1-}f(b)\sum_{k=1}^2 b^{S_{k-1}} \eta_k\1_{\{|\eta_k|>\rho b^{-S_{k-1}}((1-b^{2\delta})\log\log (1/(1-b^{2\delta})))^{-1/2}\}}=0\quad \text{a.s.}$$
We first show that, for $k\geq 3$, $b$ close to $1$ and $\varepsilon>0$ to be defined below,
\begin{equation}\label{aux52}
b^{-S_{k-1}} \Big(\frac{1}{(1-b^{2\delta})\log\log (1/(1-b^{2\delta}))}\Big)^{1/2} \geq e^{-\varepsilon} \Big(\frac{k}{\log\log k}\Big)^{1/2}\quad \text{a.s.}
\end{equation}

Let $3\leq k\leq N_\delta(b)$. The function $x\mapsto x/\log\log x$ is increasing for large $x$, whence 
$$\frac{1}{(1-b^{2\delta})\log\log (1/(1-b^{2\delta}))}\geq \frac{k}{\log\log k}.$$ Further, for $1\leq k\leq N_\delta(b)$, $$b^{-S_{k-1}}\geq e^{(-\log b) S_{k-1}}\geq e^{(-\log b) \inf_{0\leq i\leq N_\delta(b)-1}S_i}\quad\text{a.s.}$$ Since $\lim_{n\to\infty} S_n=+\infty$ a.s., we infer $|\inf_{i\geq 0}S_i|<\infty$ a.s.\ and thereupon $$\lim_{b\to 1-}(\log b) \inf_{1\leq i\leq N_\delta(b)}S_i=0\quad\text{a.s.}$$ Thus, given $\varepsilon>0$ there exists a random variable $b_\ast$ such that $b^{-S_{k-1}}\geq e^{-\varepsilon}$ whenever $3\leq k\leq N_\delta(b)$ and $b\in (b_\ast, 1)$ (of course, $b^{-S_{k-1}}\geq 1$ a.s.\ for all $k\in\mn$ provided that $\xi\geq 0$ a.s.). Thus, \eqref{aux52} does hold true in the present range of $k$.

Let $k\geq N_\delta(b)+1$. By the strong law of large numbers, $b^{S_{k-1}}\leq b^{\delta(k-1)}$ a.s.\ for $b$ close to $1$. Put $$\alpha_k(b):=b^{\delta(k-1)}\Big(\frac{k}{\log\log k}\Big)^{1/2}.$$ We claim that the sequence $(\alpha_k(b))_{k\geq N_\delta(b)+1}$ is nonincreasing. Indeed, $$\alpha_{k+1}(b)/\alpha_k(b)\leq b^\delta (1+1/k)^{1/2}\leq  b^\delta (1+1/(2k)) \leq b^\delta (1+(1-b^{2\delta})/2)\leq 1.$$ We have used $\max_{b\in [0,1]}(3b^\delta-b^{3\delta})=2$ for the last step. Hence, for $b$ close to $1$,
\begin{multline*}
b^{S_{k-1}}\Big(\frac{k}{\log\log k}\Big)^{1/2}\leq b^{\delta(k-1)} \Big(\frac{k}{\log\log k}\Big)^{1/2}\leq b^{\delta N_\delta(b)} \Big(\frac{N_\delta(b)}{\log\log N_\delta(b)}\Big)^{1/2}\\\leq b^{\delta(1-b^{2\delta})^{-1}} \Big(\frac{1}{(1-b^{2\delta})\log\log (1/(1-b^{2\delta}))}\Big)^{1/2}\leq e^\varepsilon \Big(\frac{1}{(1-b^{2\delta})\log\log (1/(1-b^{2\delta}))}\Big)^{1/2}
\end{multline*}
having utilized $\lim_{b\to 1-}b^{\delta(1-b^{2\delta})^{-1}}=e^{-1/2}$ for the last inequality. The proof of \eqref{aux52} is complete.

For $b\in (0,1)$, let $K(b)$ be positive integers satisfying $\lim_{b\to 1-}K(b)=\infty$. In view of \eqref{aux52}, for $b$ close to $1$,
\begin{equation}\label{aux1.5}
\Big|\sum_{k=3}^{K(b)}b^{S_{k-1}}\eta_k\1_{\mathcal{S}_k(b)} 
\Big|\leq e^\varepsilon \sum_{k=3}^{K(b)}|\eta_k|\1_{\{|\eta_k|>\rho e^{-\varepsilon}(k/\log\log k)^{1/2}\}}\quad \text{a.s.}
\end{equation}
and
\begin{equation}\label{aux1.6}
\sum_{k=3}^{K(b)} |\eta_k|\1_{\mathcal{S}_k(b)}\leq \sum_{k=3}^{K(b)}|\eta_k|\1_{\{|\eta_k|>\rho e^{-\varepsilon}(k/\log\log k)^{1/2}\}}\quad \text{a.s.}
\end{equation}
It is shown in the proof of Lemma 2.3 in \cite{Bovier+Picco:1993} that
\begin{equation}\label{aux012}
\sum_{k\geq 5}(k \log\log k)^{-1/2}\me (|\eta_k|\1_{\{|\eta_k|>\rho e^{-\varepsilon}(k/ \log\log k)^{1/2}\}})<\infty
\end{equation}
which particularly entails $$\sum_{k\geq 5}(k \log\log k)^{-1/2}|\eta_k|\1_{\{|\eta_k|>\rho e^{-\varepsilon}(k/ \log\log k)^{1/2}\}}<\infty\quad \text{a.s.}$$ By Kronecker's lemma, we obtain  
\begin{equation}\label{aux2}
\lim_{b\to 1-}(K(b)\log\log K(b))^{-1/2}\sum_{k=3}^{K(b)}|\eta_k|\1_{\{|\eta_k|>\rho e^{-\varepsilon}(k/\log\log k)^{1/2}\}} =0\quad\text{a.s.}
\end{equation}

We treat the sums $\sum_{k=3}^{N_\delta(b)}$ and $\sum_{k=N_\delta(b)+1}^{N_{1,\,\delta,\,\theta}(b)}$ 
separately. Relation 
\eqref{aux2} with $K(b)=N_\delta(b)$ implies that, for all $\rho>0$,
\begin{equation}\label{aux4}
\lim_{b\to 1-}f(b)\sum_{k=1}^{N_\delta(b)}b^{S_{k-1}}\eta_k\1_{\mathcal{S}_k(b)} 
=0\quad \text{a.s.}
\end{equation}
To deal with the second sum, we write, for $b$ close to $1$,
\begin{multline*}
\Big|\sum_{k=N_\delta(b)+1}^{N_{1,\,\delta,\,\theta}(b)}b^{S_{k-1}}\eta_k\1_{\mathcal{S}_k(b)}\Big|\leq \sum_{k=N_\delta(b)+1}^{N_{1,\,\delta,\,\theta}(b)}b^{\delta(k-1)}|\eta_k|\1_{\mathcal{S}_k(b)}=-b^{\delta N_\delta(b)}\sum_{k=1}^{N_\delta(b)} |\eta_k|\1_{\mathcal{S}_k(b)}\\+b^{\delta(N_{1,\,\delta,\,\theta}(b)-1)}\sum_{k=1}^{N_{1,\,\delta,\,\theta}(b)} |\eta_k|\1_{\mathcal{S}_k(b)}+(1-b^\delta)\sum_{k=N_\delta(b)+1}^{N_{1,\,\delta,\,\theta}(b)-1}b^{\delta (k-1)} \sum_{j=1}^k |\eta_j|\1_{\mathcal{S}_j(b)}=:I_1(b)+I_2(b)+I_3(b)
\end{multline*}
having utilized the strong law of large numbers for the inequality.

\noindent {\it Analysis of $I_1$}.
The limit relation $\lim_{b\to 1-}b^{\delta N_\delta(b)}=e^{-1/2}$ together with \eqref{aux1.6} and \eqref{aux2} in which we take $K(b)=N_\delta(b)$ proves $\lim_{b\to 1-}f(b)I_1(b)=0$ a.s.

\noindent {\it Analysis of $I_2$}. Using \eqref{aux1.6} and \eqref{aux2} with $K(b)=N_{1,\,\delta,\,\theta}(b)$ we infer
\begin{equation*}
\lim_{b\to 1-}(N_{1,\,\delta,\,\theta}(b)\log\log N_{1,\,\delta,\,\theta}(b))^{-1/2}\sum_{k=1}^{N_{1,\,\delta,\,\theta}(b)}|\eta_k|\1_{\mathcal{S}_k(b)}
=0\quad \text{a.s.}
\end{equation*}
Combining this with the first part of \eqref{aux55} we obtain $\lim_{b\to 1-} f(b)I_2(b)=0$ a.s.

\noindent {\it Analysis of $I_3$}. Write
\begin{multline*}
I_3(b)\leq (1-b^\delta) (\sup_{N_\delta(b)+1\leq k\leq N_{1,\,\delta,\,\theta}(b)-1}\,T_k(b))\sum_{k\geq 3} b^{\delta(k-1)}(k\log\log k)^{1/2}\\\sim
(\pi^{1/2}/2)\Big(\frac{\log\log (1/(1-b^\delta))}{1-b^\delta}\Big)^{1/2} (\sup_{N_\delta(b)+1\leq k\leq N_{1,\,\delta,\,\theta}(b)-1}\,T_k(b)),\quad b\to 1-,
\end{multline*}
where $$T_k(b):=(k\log\log k)^{-1/2}\sum_{j=1}^k |\eta_j|\1_{\mathcal{S}_j(b)}.$$ We have used Corollary 1.7.3 in \cite{Bingham+Goldie+Teugels:1989} for the asymptotic equivalence. In view of \eqref{aux2} with $K(b)=N_\delta(b)+1$ and \eqref{aux1.6}, $\lim_{b\to 1-}\sup_{N_\delta(b)+1\leq k\leq N_{1,\,\delta\,\theta}(b)-1}\,T_k(b)=0$ a.s., whence $\lim_{b\to 1-} f(b)I_3(b)=0$ a.s. The proof of \eqref{aux010} is complete.

\noindent {\sc Proof of \eqref{aux011}}. Similarly to \eqref{aux4}, we obtain with the help of
\begin{equation*}
\Big|\sum_{k=3}^{K(b)}b^{S_{k-1}}\me_{k-1}(\eta_k\1_{\mathcal{S}_k(b)})
\Big|\leq e^\varepsilon \sum_{k=3}^{K(b)}\me |\eta_k|\1_{\{|\eta_k|>\rho e^{-\varepsilon}(k/\log\log k)^{1/2}\}}\quad \text{a.s.}
\end{equation*}
(a counterpart of \eqref{aux1.5}) and \eqref{aux012} that
$$\lim_{b\to 1-}f(b)\sum_{k=1}^{N_\delta(b)}b^{S_{k-1}}\me_{k-1}(\eta_k\1_{\mathcal{S}_k(b)})=0\quad\text{a.s.}$$ By the same reasoning, we also conclude that $\lim_{b\to 1-} f(b)I^\ast_\ell(b)=0$ a.s., $\ell=1,2,3$, where $I_\ell^\ast(b)$ is a counterpart of $I_\ell(b)$ in which $|\eta_k|\1_{\mathcal{S}_k(b)}$ is replaced with $\me_{k-1}(|\eta_k|\1_{\mathcal{S}_k(b)})$.

The proof of Lemma \ref{3} is complete.
\end{proof}

As usual, $\mathcal{S}^c_k(b)$ will denote the complement of $\mathcal{S}_k(b)$, that is,
$$\mathcal{S}^c_k(b)=\{|\eta_k|\leq \rho b^{-S_{k-1}}((1-b^{2\delta})\log\log (1/(1-b^{2\delta})))^{-1/2}\}.$$ Denote by $\mathbb{B}$ the class of increasing sequences $(\tb_n)_{n\in\mn}$ of positive numbers satisfying the following properties:

\noindent (a) $\lim_{n\to\infty}\tb_n=1$ and $\lim_{n\to\infty}\frac{1-\tb_{n+1}}{1-\tb_n}=1$;

\noindent (b) $\lim_{n\to\infty} \frac{\tb_{n+1}-\tb_n}{1-\tb_n}\Big(\log\log \frac{1}{1-\tb_n}\Big)^{3/2}=0$;

\noindent (c) for all $\varepsilon>0$, $\sum_{n\geq 1}\Big(\log\Big(\frac{1}{1-\tb_n}\Big)\Big)^{-1-\varepsilon}<\infty$.

One can check that any increasing sequence $(\texttt{b}_n)_{n\in\mn}$ of positive numbers satisfying $\tb_n=\exp(-(1-(\log n)^{-3})^n)$ for large $n$ belongs to the class $\mathbb{B}$. For instance, for the so defined $\tb_n$ we have $$\frac{\tb_{n+1}-\tb_n}{1-\tb_n}\sim (\log n)^{-3}\quad\text{and}\quad \log\log \frac{1}{1-\tb_n}\sim \log n,\quad n\to\infty$$ which verifies the property (b).

Recall that `i.o.' is a shorthand for `infinitely often' and that, for a sequence of sets $A_1$, $A_2,\ldots$, $$\{A_n ~\text{{\rm i.o.}}\}:=\{\cup_{n\geq 1}\cap_{k\geq n}A_k\}.$$
\begin{lemma}\label{4}
Let $(\tb_n)_{n\in\mn}\in \mathbb{B}$. Then, 
for all $\varepsilon>0$, $$\mmp\Big\{\sum_{k=1}^{N_{1,\,\delta,\,\theta}(\tb_n)}\tb_n^{S_{k-1}}\tilde \eta_k(\tb_n)>\big((2+\varepsilon)N(\tb_n)\log\log N(\tb_n)\big)^{1/2}\quad\text{{\rm i.o.}}\Big\}=0,$$
where $\tilde \eta_k(b):=\eta_k\1_{\mathcal{S}^c_k(b)}-\me_{k-1}(\eta_k\1_{\mathcal{S}^c_k(b)})$ for $k\in\mn$ and $N(b):=(1-b^2)^{-1}$ for $b\in (0,1)$.
\end{lemma}
\begin{proof}
The proof below follows the path of the proof of Lemma 3.6 in \cite{Picco+Vares:1994}.

We start by showing that
\begin{equation}\label{aux30}
{\lim\sup}_{b\to 1-}(1-b^{2\delta})\sum_{k=1}^{N_{1,\,\delta,\,\theta}(b)}b^{2S_{k-1}}\me_{k-1}(\tilde \eta_k^2(b))\leq \delta\quad\text{a.s.}
\end{equation}
(recall that $\mu=1$ by convention). Indeed, $\me_{k-1}(\tilde \eta_k^2(b))\leq \me \eta_k^2=1$, whence $$\sum_{k=1}^{N_{1,\,\delta,\,\theta}(b)}b^{2S_{k-1}}\me_{k-1}(\tilde \eta_k^2(b))\leq \sum_{k\geq 0} b^{2S_k}\quad\text{a.s.}$$ By Theorem \ref{main1} with $\eta=1$ a.s., $$\lim_{b\to 1-}(1-b^{2\delta})\sum_{k\geq 0} b^{2S_k}=\delta\quad\text{a.s.}$$ which entails \eqref{aux30}.

For $n\in\mn$, put $$t_n:=\Big(\frac{(2+\varepsilon)\log\log N(\tb_n)}{N(\tb_n)}\Big)^{1/2}$$ and define the event $$\mathcal{B}_n:=\Big\{\sum_{k=1}^{N_{1,\,\delta,\,\theta}(\tb_n)}\tb_n^{S_{k-1}}\tilde \eta_k(\tb_n)>\big((2+\varepsilon)N(\tb_n)\log\log N(\tb_n)\big)^{1/2}\Big\}.$$ Equivalently,
\begin{multline*}
\mathcal{B}_n:=\Big\{t_n\sum_{k=1}^{N_{1,\,\delta,\,\theta}(\tb_n)}\tb_n^{S_{k-1}}\tilde \eta_k(\tb_n)-(t_n^2/2) e^{4\rho(1+\varepsilon)}\sum_{k=1}^{N_{1,\,\delta,\,\theta}(\tb_n)}\tb_n^{2S_{k-1}}\me_{k-1}(\tilde \eta^2_k(\tb_n))\\> t_n^2 N(\tb_n)\Big(1-(e^{4\rho(1+\varepsilon)}/(2N(\tb_n)))\sum_{k=1}^{N_{1,\,\delta,\,\theta}(\tb_n)}\tb_n^{2S_{k-1}}\me_{k-1}(\tilde \eta^2_k(\tb_n))\Big)\Big\}.
\end{multline*}
In view of \eqref{aux30}, given $\beta>0$,
\begin{multline*}
\mathcal{B}_n\subseteq \mathcal{A}_n:=\Big\{t_n\sum_{k=1}^{N_{1,\,\delta,\,\theta}(\tb_n)}\tb_n^{S_{k-1}}\tilde \eta_k(\tb_n)-(t_n^2/2) e^{4\rho(1+\varepsilon)}\sum_{k=1}^{N_{1,\,\delta,\,\theta}(\tb_n)}\tb_n^{2S_{k-1}}\me_{k-1}(\tilde \eta^2_k(\tb_n))\\>t_n^2 N(\tb_n)(1-(e^{4\rho(1+\varepsilon)}/2)(\delta+\beta))\Big\}
\end{multline*}
for large enough $n$. Thus, by the Borel-Cantelli lemma, Lemma \ref{4} 
follows if we can check that
\begin{equation}\label{aux32}
\sum_{n\geq 1}\mmp(\mathcal{A}_n)<\infty.
\end{equation}

As a preparation for this matter, we intend to show that
\begin{equation}\label{aux31}
\me\tau_n=:\me\exp\Big(t_n\sum_{k=1}^{N_{1,\,\delta,\,\theta}(\tb_n)}\tb_n^{S_{k-1}}\tilde \eta_k(\tb_n)-(t_n^2/2) e^{4\rho(1+\varepsilon)}\sum_{k=1}^{N_{1,\,\delta,\,\theta}(\tb_n)}\tb_n^{2S_{k-1}}\me_{k-1}(\tilde \eta^2_k(\tb_n))\Big)\leq 1.
\end{equation}
Using $e^x\leq 1+x+(x^2/2)e^{|x|}$ for $x\in\mr$ and $\me_{k-1}\tilde \eta_k(\tb_n)=0$ we infer $$\me_{k-1}\exp(t_n \tb_n^{S_{k-1}}\tilde \eta_k(\tb_n))\leq 1+(t_n^2/2)\tb_n^{2S_{k-1}}\me_{k-1}(\tilde \eta^2_k(\tb_n)\exp(t_n\tb_n^{S_{k-1}}|\tilde \eta_k(\tb_n)|))\quad\text{a.s.}$$ Further, $$t_n \tb_n^{S_{k-1}}|\tilde \eta_k(\tb_n)|\leq 2\rho (2+\varepsilon)^{1/2}\leq 4\rho(1+\varepsilon)\quad\text{a.s.}$$ This in combination with $e^x\geq 1+x$ for $x\geq 0$ yields  $$\me_{k-1}\big(\exp(t_n \tb_n^{S_{k-1}}\tilde \eta_k(\tb_n))\big)\exp(-(t_n^2/2)e^{4\rho(1+\varepsilon)}\tb_n^{2S_{k-1}}\me_{k-1}(\tilde \eta^2_k(\tb_n))\leq 1\quad\text{a.s.}$$ Inequality \eqref{aux31} is a consequence of this and the tower property of conditional expectations:
\begin{multline*}
\me_{N_{1,\,\delta,\,\theta}(\tb_n)-1}\tau_n\\=\exp\Big(t_n\sum_{k=1}^{N_{1,\,\delta,\,\theta}(\tb_n)-1}\tb_n^{S_{k-1}}\tilde \eta_k(\tb_n)-(t_n^2/2) e^{4\rho(1+\varepsilon)}\sum_{k=1}^{N_{1,\,\delta,\,\theta}(\tb_n)-1}\tb_n^{2S_{k-1}}\me_{k-1}(\tilde \eta^2_k(\tb_n))\Big)\\\times \me_{N_{1,\,\delta,\,\theta}(\tb_n)-1}\big(\exp(t_n \tb_n^{S_{k-1}}\tilde \eta_{N_{1,\,\delta,\,\theta}(\tb_n)}(\tb_n))\big)\exp(-(t_n^2/2)e^{4\rho(1+\varepsilon)}\tb_n^{2S_{k-1}}\me_{N_{1,\,\delta,\,\theta}(\tb_n)-1}(\tilde \eta^2_{N_{1,\,\delta,\,\theta}(\tb_n)}(\tb_n))\\\leq \exp\Big(t_n\sum_{k=1}^{N_{1,\,\delta,\,\theta}(\tb_n)-1}\tb_n^{S_{k-1}}\tilde \eta_k(\tb_n)-(t_n^2/2) e^{4\rho(1+\varepsilon)}\sum_{k=1}^{N_{1,\,\delta,\,\theta}(\tb_n)-1}\tb_n^{2S_{k-1}}\me_{k-1}(\tilde \eta^2_k(\tb_n))\Big)\quad\text{a.s.}
\end{multline*}
Further,
\begin{multline*}
\me_{N_{1,\,\delta,\,\theta}(\tb_n)-2}\tau_n =\me_{N_{1,\,\delta,\,\theta}(\tb_n)-2}(\me_{N_{1,\,\delta,\,\theta}(\tb_n)-1}\tau_n)\leq \exp\Big(t_n\sum_{k=1}^{N_{1,\,\delta,\,\theta}(\tb_n)-2}\tb_n^{S_{k-1}}\tilde \eta_k(\tb_n)\\-(t_n^2/2) e^{4\rho(1+\varepsilon)}\sum_{k=1}^{N_{1,\,\delta,\,\theta}(\tb_n)-2}\tb_n^{2S_{k-1}}\me_{k-1}(\tilde \eta^2_k(\tb_n))\Big)\quad\text{a.s.}
\end{multline*}
Repeating this argument $N_{1,\,\delta,\,\theta}(\tb_n)$ times we arrive at \eqref{aux31}.

We are ready to prove \eqref{aux32}. By Markov's inequality and \eqref{aux31}, $$\mmp(\mathcal{A}_n)\leq \exp\big(-t_n^2 N(\tb_n)(1-(e^{4\rho(1+\varepsilon)}/2)(\delta+\beta))\big)\me\tau_n\leq
\exp\big(-t_n^2 N(\tb_n)(1-(e^{4\rho(1+\varepsilon)}/2)(\delta+\theta))\big).$$ Given small enough $\varepsilon>0$, and $\delta\in (0,1)$ and $\beta>0$ satisfying $\delta+\beta\in (0,1)$ we can find $\rho>0$ such that $(2+\varepsilon)(1-(e^{4\rho(1+\varepsilon)/2})(\delta+\beta))>1$. This together with the property (c) of $\mathbb{B}$ ensures $$\sum_{n\geq 1}\exp(-t_n^2 N(\tb_n)((1-(e^{4\rho(1+\varepsilon)}/2)(\delta+\theta))))<\infty,$$ and \eqref{aux32} follows. The proof of Lemma \ref{4} is complete.
\end{proof}
\begin{lemma}\label{5}
Let $(\tb_n)_{n\in\mn}\in \mathbb{B}$. Then $$\lim_{n\to\infty}\sup_{b\in [\tb_n,\, \tb_{n+1}]}\Big|f(b)\sum_{k=1}^{N_{1,\,\delta,\,\theta}(b)}b^{S_{k-1}}\eta_k-f(\tb_n)\sum_{k=1}^{N_{1,\,\delta,\,\theta}(\tb_n)}\tb_n^{S_{k-1}}\eta_k \Big|=0\quad\text{{\rm a.s.}}$$
\end{lemma}
\begin{proof}
Throughout the proof we tacitly assume that the equalities and inequalities hold a.s. We start by writing, for $b\in [\tb_n, \tb_{n+1}]$,
\begin{multline*}
f(b)\sum_{k=1}^{N_{1,\,\delta,\,\theta}(b)}b^{S_{k-1}}\eta_k-f(\tb_n)\sum_{k=1}^{N_{1,\,\delta,\,\theta}(\tb_n)}\tb_n^{S_{k-1}}\eta_k=
\sum_{k=1}^{N_{1,\,\delta,\,\theta}(\tb_n)}\big(f(b)b^{S_{k-1}}-f(\tb_n)\tb_n^{S_{k-1}}\big)\eta_k\\+f(b)\sum_{k=N_{1,\,\delta,\,\theta}(\tb_n)+1}^{N_{1,\,\delta,\,\theta}(b)}b^{S_{k-1}}\eta_k=:I_n(b)+J_n(b).
\end{multline*}
Summation by parts yields
\begin{multline*}
I_n(b)=\big(f(b)b^{S_{N_{1,\,\delta,\,\theta}(\tb_n)-1}}-f(\tb_n)\tb_n^{S_{N_{1,\,\delta,\,\theta}(\tb_n)-1}}\big)T_{N_{1,\,\delta,\,\theta}(\tb_n)}\\+\sum_{k=1}^{N_{1,\,\delta,\,\theta}(\tb_n)-1}\big(f(b)(b^{S_{k-1}}-b^{S_k})-
f(\tb_n)(\tb_n^{S_{k-1}}-\tb_n^{S_k})\big)T_k=:
I_{n,1}(b)+I_{n,2}(b),
\end{multline*}
where $T_k=\eta_1+\ldots+\eta_k$ for $k\in\mn$. For large enough $n$ for which $S_{N_{1,\,\delta,\,\theta}(\tb_n)-1}\geq \delta (N_{1,\,\delta,\,\theta}(\tb_n)-1)$
a.s. (this is secured by the strong law of large numbers) and, given $\varepsilon>0$,
\begin{equation}\label{aux21}
\delta|\log \tb_{n+1}|/(1-\tb_n^{2\delta})\geq 1/2-\varepsilon
\end{equation} 
(this is ensured by the property (a) of $\mathbb{B}$),
\begin{eqnarray*}
&&\sup_{b\in [\tb_n,\, \tb_{n+1}]}|I_{n,1}(b)|\leq 2f(\tb_n)\tb_{n+1}^{\delta (N_{1,\,\delta,\,\theta}(\tb_n)-1)} |T_{N_{1,\,\delta,\,\theta}(\tb_n)}|\notag\\&\leq& 2\tb_1^{-2\delta}f(\tb_n)(\log(1/(1-\tb_n^{2\delta})))^{-(1/2-\varepsilon)(1+\theta)}O\big((N_{1,\,\delta,\,\theta}(\tb_n)\log\log N_{1,\,\delta,\,\theta}(\tb_n))^{1/2}\big)\\&=&O\Big(\frac{(\log\log (1/(1-\tb_n)))^{1/2}}{(\log (1/(1-\tb_n)))^{(1/2-\varepsilon)(1+\theta)}}\Big)~\to~0\quad\text{a.s.\ as}~~n\to\infty.
\end{eqnarray*}
having utilized \eqref{lilweak} for the inequality. We are now passing to the analysis of $I_{n,2}(b)$. By the strong law of large numbers, with the same $\delta\in (0,1)$ there exists an a.s.\ finite $\tau$ such that $\max(\delta k, 1)\leq S_k\leq (2-\delta)k$ for all $k\geq \tau+1$. Since, for $b\in [\tb_n,\, \tb_{n+1}]$, $$\Big|\sum_{k=1}^{\tau}f(b)(b^{S_{k-1}}-b^{S_k})T_k\Big|\leq f(b) \sum_{k=1}^{\tau}(b^{S_{k-1}}+b^{S_k})|T_k|\leq f(\tb_n)\sum_{k=1}^{\tau}(\tb_n^{S_{k-1}}+\tb_n^{S_k})|T_k| 
$$ we infer
$$\lim_{n\to\infty}\sup_{b\in [\tb_n\,\tb_{n+1}]}\Big|\sum_{k=1}^{\tau}f(b)\big(b^{S_{k-1}}-b^{S_k}\big)T_k\Big|=0\quad\text{a.s.}$$ We need some preparation to treat the remaining part of the sum.
Using the fact that when $\xi_k\geq 0$ the function $b\mapsto f(b)(1-b^{\xi_k})$ is nonincreasing for $b<1$ close to $1$ we obtain on the event $\{\xi_k\geq 0, \tau\leq k-1\}$, for $b\in [\tb_n,\,\tb_{n+1}]$ and large $n\in\mn$,
\begin{multline*}
f(\tb_n)(\tb_n^{S_{k-1}}-\tb_{n+1}^{S_{k-1}})(1-\tb_n^{\xi_k})\leq f(\tb_n)\big(\tb_n^{S_{k-1}}-\tb_n^{S_k}\big)-f(b)\big(b^{S_{k-1}}-b^{S_k}\big)\\\leq \tb_n^{S_{k-1}}\big(f(\tb_n)(1-\tb_n^{\xi_k})-f(\tb_{n+1})(1-\tb_{n+1}^{\xi_k})\big). 
\end{multline*}
Combining this with a similar inequality on the event $\{\xi_k<0, \tau\leq k-1\}$ we arrive at
\begin{multline*}
\big|f(\tb_n)\big(\tb_n^{S_{k-1}}-\tb_n^{S_k}\big)-f(b)\big(b^{S_{k-1}}-b^{S_k}\big)\big|\leq f(\tb_n)\big(\tb_{n+1}^{S_{k-1}}-\tb_n^{S_{k-1}}\big)|1-\tb_n^{\xi_k}|\\+\tb_n^{S_{k-1}}\big(f(\tb_n)|1-\tb_n^{\xi_k}|-f(\tb_{n+1})|1-\tb_{n+1}^{\xi_k}|\big)
\end{multline*}
for $b$ and $n$ as above. Thus, for $b\in [\tb_n,\,\tb_{n+1}]$ and large $n\in\mn$,
\begin{eqnarray*}
&&\Big|\sum_{k=\tau+1}^{N_{1,\,\delta,\,\theta}(\tb_n)-1}\big(f(b)(b^{S_{k-1}}-b^{S_k})-
f(\tb_n)(\tb_n^{S_{k-1}}-\tb_n^{S_k})\big)T_k\Big|\\&\leq& \sum_{k=\tau+1}^{N_{1,\,\delta,\,\theta}(\tb_n)-1}\big|f(b)(b^{S_{k-1}}-b^{S_k})-
f(\tb_n)(\tb_n^{S_{k-1}}-\tb_n^{S_k})\big||T_k|\\&\leq& f(\tb_n)\sum_{k=\tau+1}^{N_{1,\,\delta,\,\theta}(\tb_n)-1}\big(\tb_{n+1}^{S_{k-1}}-\tb_n^{S_{k-1}}\big)|1-\tb_n^{\xi_k}||T_k|\\&+&
\sum_{k=\tau+1}^{N_{1,\,\delta,\,\theta}(\tb_n)-1} \tb_n^{S_{k-1}}\big(f(\tb_n)|1-\tb_n^{\xi_k}|-f(\tb_{n+1})|1-\tb_{n+1}^{\xi_k}|\big)|T_k|=:I_{n,21}(b)+I_{n,22}(b). 
\end{eqnarray*}
For all $k\in\mn$ and all $n\in\mn$,
\begin{equation}\label{aux567}
|1-\tb_n^{\xi_k}|\leq |\log \tb_n|(\xi_k^++\xi_k^-\tb_n^{\xi_k}),
\end{equation}
where, as usual, $x^+=\max (x,0)$ and $x^-=\max (-x, 0)$ for $x\in\mr$. For $k\geq \tau+1$ and $n\in\mn$, by the mean value theorem for differentiable functions,
\begin{equation}\label{aux569}
\tb_{n+1}^{S_{k-1}}-\tb_n^{S_{k-1}}\leq S_{k-1}\tb_{n+1}^{S_{k-1}-1}(\tb_{n+1}-\tb_n)\leq \tb_1^{-1}S_{k-1}\tb_{n+1}^{S_{k-1}}(\tb_{n+1}-\tb_n)
\end{equation}
and thereupon
$$(\tb_{n+1}^{S_{k-1}}-\tb_n^{S_{k-1}})|1-\tb_n^{\xi_k}|\leq  (2-\delta) \tb_1^{-1}(\tb_{n+1}-\tb_n)|\log \tb_n|k\big(\tb_{n+1}^{S_{k-1}}\xi_k^++(\tb_{n+1}/\tb_n)^{S_{k-1}}\tb_n^{S_k}\xi_k^-\big).$$ Thus,
\begin{eqnarray*}
&&(2-\delta)^{-1}\tb_1 \sum_{k=\tau+1}^{N_{1,\,\delta,\,\theta}(\tb_n)-1}\big(\tb_{n+1}^{S_{k-1}}-\tb_n^{S_{k-1}}\big)\big|1-\tb_n^{\xi_k}\big| |T_k|\\&\leq& (\tb_{n+1}-\tb_n)|\log \tb_n|  \sum_{k=1}^{N_{1,\,\delta,\,\theta}(\tb_n)}k|T_k| \big(\tb_{n+1}^{\delta(k-1)}\xi_k^++(\tb_{n+1}/\tb_n)^{(2-\delta)k}\tb_n^{\delta k}\xi_k^-\big)\\&\leq& (\tb_{n+1}-\tb_n)|\log \tb_n|N_{1,\,\delta,\,\theta}(\tb_n)(\sup_{1\leq k\leq N_{1,\,\delta,\,\theta}(\tb_n)}\,|T_k|)\Big(\sum_{k\geq 1}\tb_{n+1}^{\delta(k-1)}\xi_k^+\\&+& (\tb_{n+1}/\tb_n)^{(2-\delta)N_{1,\,\delta,\,\theta}(\tb_n)} \sum_{k\geq 1}\tb_n^{\delta k}\xi_k^-\Big).
\end{eqnarray*}
By Theorem \ref{main1}, as $n\to\infty$,
\begin{equation}\label{aux568}
\sum_{k\geq 1}\tb_{n+1}^{\delta(k-1)}\xi_k^+~\sim~ (1-\tb_{n+1}^\delta)^{-1}\me\xi^+\quad\text{and}\quad \sum_{k\geq 1}\tb_n^{\delta k}\xi_k^-~\sim~ (1-\tb_n^\delta)^{-1}\me \xi^-\quad\text{a.s.}
\end{equation}
Using \eqref{lilweak} 
in combination with the property (a) of $\mathbb{B}$ for the first equality and the property (b) of $\mathbb{B}$ for the second we infer
\begin{multline*}
f(\tb_n)(\tb_{n+1}-\tb_n)|\log \tb_n|N_{1,\,\delta,\,\theta}(\tb_n)(\sup_{1\leq k\leq N_{1,\,\delta,\,\theta}(\tb_n)}\,|T_k|) \sum_{k\geq 1}\tb_{n+1}^{\delta(k-1)}\xi_k^+\\=O\Big(\frac{\tb_{n+1}-\tb_n}{1-\tb_n}\Big(\log\log \frac{1}{1-\tb_n}\Big)^{3/2}\Big)=o(1)\quad\text{a.s.~as}~~n\to\infty.
\end{multline*}
Invoking once again the property (b) of $\mathbb{B}$ we obtain $\lim_{n\to\infty}\log (\tb_{n+1}/\tb_n)N_{1,\,\delta,\,\theta}(\tb_n)=0$, whence $\lim_{n\to\infty}(\tb_{n+1}/\tb_n)^{(2-\delta)N_{1,\,\delta,\,\theta}(\tb_n)}=1$. With this at hand we can argue as before to conclude that a.s.
\begin{equation*}
\lim_{n\to\infty} f(\tb_n)(\tb_{n+1}-\tb_n)|\log \tb_n|N_{1,\,\delta,\,\theta}(\tb_n)(\tb_{n+1}/\tb_n)^{(2-\delta)N_{1,\,\delta,\,\theta}(\tb_n)}(\sup_{1\leq k\leq N_{1,\,\delta,\,\theta}(\tb_n)}\,|T_k|) \sum_{k\geq 1}\tb_n^{\delta k}\xi_k^-=0. 
\end{equation*}
Thus, we have proved that $\lim_{n\to\infty}I_{n,21}(b)=0$ a.s.

Further,
\begin{multline*} I_{n,22}(b)=(f(\tb_n)-f(\tb_{n+1}))\sum_{k=\tau+1}^{N_{1,\,\delta,\,\theta}(\tb_n)-1}\tb_n^{S_{k-1}}|1-\tb_n^{\xi_k}||T_k|\\+f(\tb_{n+1})\sum_{k=\tau+1}^{N_{1,\,\delta,\,\theta}(\tb_n)-1} \tb_n^{S_{k-1}}\big(\big|1-\tb_n^{\xi_k}\big|-\big|1-\tb_{n+1}^{\xi_k}\big|\big)|T_k|.
\end{multline*}
In view of \eqref{aux567},
\begin{multline*}
\sum_{k=\tau+1}^{N_{1,\,\delta,\,\theta}(\tb_n)-1}\tb_n^{S_{k-1}}|1-\tb_n^{\xi_k}||T_k|\leq |\log \tb_n|(\sup_{1\leq k\leq N_{1,\,\delta,\,\theta}(\tb_n)}\,|T_k|)
\sum_{k=\tau+1}^{N_{1,\,\delta,\,\theta}(\tb_n)-1}(\tb_n^{S_{k-1}}\xi_k^++\tb_n^{S_k}\xi_k^-)\\\leq |\log \tb_n|(\sup_{1\leq k\leq N_{1,\,\delta,\,\theta}(\tb_n)}\,|T_k|)
\sum_{k\geq 1}(\tb_n^{\delta(k-1)}\xi_k^++\tb_n^{\delta k}\xi_k^-).
\end{multline*}
Invoking \eqref{aux568} and \eqref{lilweak} in combination with $\lim_{n\to\infty}N_{1,\,\delta,\,\theta}(\tb_n)(f(\tb_n))^2=(1+\theta)(2\delta)^{-1}$ we conclude that
\begin{multline*}
(f(\tb_n)-f(\tb_{n+1}))\sum_{k=\tau+1}^{N_{1,\,\delta,\,\theta}(\tb_n)-1}\tb_n^{S_{k-1}}|1-\tb_n^{\xi_k}||T_k|=O\Big(\frac{f(\tb_n)-f(\tb_{n+1})}{f(\tb_n)}\Big(\log\log\frac{1}{1-\tb_n^2}\Big)^{1/2}\Big)\\
=o(1)\quad\text{a.s.~as}~~n\to\infty.
\end{multline*}
The last equality is justified as follows. Using subadditivity of $x\mapsto x^{1/2}$ on $[0,\infty)$ we obtain, for large $n$,
\begin{multline*}
\Big(\frac{f(\tb_n)-f(\tb_{n+1})}{f(\tb_n)}\Big)^2\log\log\frac{1}{1-\tb_n^2}\\ \leq \Big(\log\log\frac{1}{1-\tb_{n+1}^2}-\log\log\frac{1}{1-\tb_n^2}+\frac{\tb_{n+1}^2-\tb_n^2}{1-\tb_n^2}\log\log\frac{1}{1-\tb_n^2}\Big)\frac{\log\log\frac{1}{1-\tb_n^2}}{\log\log\frac{1}{1-\tb_{n+1}^2}}.
\end{multline*}
The property (a) of $\mathbb{B}$ entails
\begin{equation}\label{aux54}
\lim_{n\to\infty}\frac{\log(1/(1-\tb^2_{n+1})) }{\log(1/(1-\tb^2_n))}=1\quad\text{and}\quad \lim_{n\to\infty}\frac{\log\log(1/(1-\tb^2_{n+1})) }{\log\log (1/(1-\tb^2_n))}=1,
\end{equation}
and the first of these ensures $$\lim_{n\to\infty}\Big(\log\log\frac{1}{1-\tb_{n+1}^2}-\log\log\frac{1}{1-\tb_n^2}\Big)=0.$$ Finally, $$\lim_{n\to\infty} \frac{\tb_{n+1}^2-\tb_n^2}{1-\tb_n^2}\log\log\frac{1}{1-\tb_n^2}=0$$ is a consequence of the property (b) of $\mathbb{B}$. Thus, the equality that we wanted to justify does indeed hold.

For the analysis of the second piece of $I_{n,22}(b)$ we need an estimate similar to \eqref{aux569}: for $k,n\in\mn$, $$\big|1-\tb_n^{\xi_k}\big|-\big|1-\tb_{n+1}^{\xi_k}\big|=\big|\tb_{n+1}^{\xi_k}-\tb_n^{\xi_k}\big|\leq (\tb_{n+1}-\tb_n)\big(\xi_k^+ \tb_{n+1}^{\xi_k-1}+\xi_k^-\tb_n^{\xi_k-1}\big)\leq \tb_1^{-1} (\tb_{n+1}-\tb_n)\big(\xi_k^++\xi_k^- \tb_n^{\xi_k}\big). 
$$ This implies that
\begin{multline*}
f(\tb_{n+1})\sum_{k=\tau+1}^{N_{1,\,\delta,\,\theta}(\tb_n)-1} \tb_n^{S_{k-1}}\big(\big|1-\tb_n^{\xi_k}\big|-\big|1-\tb_{n+1}^{\xi_k}\big|\big)|T_k|\\\leq \tb_1^{-1}f(\tb_n)(\tb_{n+1}-\tb_n)(\sup_{1\leq k\leq N_{1,\,\delta,\,\theta}(\tb_n)}\,|T_k|)\sum_{k\geq 1}(\tb_n^{\delta(k-1)}\xi_k^++\tb_n^{\delta k}\xi_k^-)\\=O\Big(\frac{\tb_{n+1}-\tb_n}{1-\tb_n}\Big(\log\log \frac{1}{1-\tb_n}\Big)^{1/2}\Big)=o(1)\quad\text{a.s.~as}~~n\to\infty.
\end{multline*}
Here, while the first equality is ensured by \eqref{lilweak} and \eqref{aux568}, the second is a consequence of the property (b) of $\mathbb{B}$. The proof of $\lim_{n\to\infty}I_n (b)=0$ a.s.\ is complete.

We proceed by analyzing $J_n(b)$: for $b\in [\tb_n,\,\tb_{n+1}]$,
\begin{multline*}
J_n(b)=f(b)\sum_{k=N_{1,\,\delta,\,\theta}(\tb_n)}^{N_{1,\,\delta,\,\theta}(b)-1}(b^{S_{k-1}}-b^{S_k})T_k+f(b)(b^{S_{N_{1,\,\delta,\,\theta}(b)-1}}T_{N_{1,\,\delta,\,\theta}(b)}
-b^{S_{N_{1,\,\delta,\,\theta}(\tb_n)-1}}T_{N_{1,\,\delta,\,\theta}(\tb_n)})\\=:
J_{n,1}(b)+J_{n,2}(b).
\end{multline*}
As before, appealing to the strong law of large numbers, we conclude that
\begin{multline*}
\sup_{b\in [\tb_n,\,\tb_{n+1}]}|J_{n,2}(b)|\leq 2f(\tb_n)\tb_{n+1}^{\delta(N_{1,\,\delta,\,\theta}(\tb_n)-1)}\sup_{k\leq N_{1,\,\delta,\,\theta}(\tb_{n+1})}\,|T_k|\\\leq 2\tb_1^{-2\delta}f(\tb_n)(\log(1/(1-\tb_n^{2\delta})))^{-(1/2-\varepsilon)(1+\theta)}O\big((N_{1,\,\delta,\,\theta}(\tb_{n+1})\log\log N_{1,\,\delta,\,\theta}(\tb_{n+1}))^{1/2}\big)\\=O\Big(\frac{(\log\log (1/(1-\tb_n)))^{1/2}}{(\log (1/(1-\tb_n)))^{(1/2-\varepsilon)(1+\theta)}}\Big)~\to~0\quad\text{a.s.\ as}~~n\to\infty.
\end{multline*}
We have used \eqref{lilweak} and \eqref{aux21} for the inequality and the property (a) of $\mathbb{B}$ and its consequences \eqref{aux54} for the equality. Invoking \eqref{aux53} we obtain, for large $n$ and appropriate constant $C>0$,
\begin{multline*}
\sup_{b\in [\tb_n,\,\tb_{n+1}]}|J_{n,1}(b)|\leq \sup_{b\in [\tb_n,\,\tb_{n+1}]} f(b)\sum_{k=N_{1,\,\delta,\,\theta}(\tb_n)}^{N_{1,\,\delta,\,\theta}(b)-1}|b^{S_{k-1}}-b^{S_k}||T_k|\\\leq Cf(\tb_n)|\log \tb_n|\sum_{k=N_{1,\,\delta,\,\theta}(\tb_n)}^{N_{1,\,\delta,\,\theta}(\tb_{n+1})-1}\tb_{n+1}^{\delta k}|\xi_k|(\sup_{k\leq N_{1,\,\delta,\,\theta}(\tb_{n+1})}\,|T_k|)\\ \leq Cf(\tb_n)|\log \tb_n|\sum_{k\geq N_{1,\,\delta,\,\theta}(\tb_n)} \tb_{n+1}^{\delta k}|\xi_k|O\big((N_{1,\,\delta,\,\theta}(\tb_{n+1})\log\log N_{1,\,\delta,\,\theta}(\tb_{n+1}))^{1/2}\big).
\end{multline*}
We use Lemma \ref{stronglaw} with $\eta=|\xi|$, $\lambda=\delta$, $M(b)=N_{1,\,\delta,\,\theta}(b)-1$, $x_n=\tb_n$ and $y_n=\tb_{n+1}$. Recalling the property (a) of $\mathbb{B}$ we conclude that  $\lim_{n\to\infty}N_{1,\,\delta,\,\theta}(\tb_n)(1-\tb_{n+1}^\delta)=\infty$. Hence, an application of that lemma yields $$\sum_{k\geq N_{1,\,\delta,\,\theta}(\tb_n)} \tb_{n+1}^{\delta k}|\xi_k|~\sim~ \me |\xi|\tb_{n+1}^{N_{1,\,\delta,\,\theta}(\tb_n)}(1-\tb^\delta_{n+1})^{-1}\quad\text{a.s.\ as}~~n\to\infty.$$ Using once again the property (a) of $\mathbb{B}$ and \eqref{aux54} in combination with the estimate for $\tb_{n+1}^{N_{1,\,\delta,\,\theta}(\tb_n)}$ which is implied by \eqref{aux21} we infer $$\sup_{b\in [\tb_n,\,\tb_{n+1}]}|J_{n,1}(b)|=O\Big(\frac{(\log\log (1/(1-\tb_n)))^{1/2}}{(\log (1/(1-\tb_n)))^{(1/2-\varepsilon)(1+\theta)}}\Big)~\to~0\quad\text{a.s.\ as}~~n\to\infty.$$ The proof of Lemma \ref{5} is complete.
\end{proof}

We are ready to prove Proposition \ref{lilhalf}.
\begin{proof}[Proof of Proposition \ref{lilhalf}]
We only prove \eqref{lil1}, for \eqref{lil2} is a consequence of \eqref{lil1}  with $-\eta_k$ replacing $\eta_k$.

By Lemmas \ref{1} and \ref{2} and \eqref{aux010}, \eqref{lil1} is equivalent to $${\lim\sup}_{b\to 1-}f(b)\sum_{k=1}^{N_{1,\,\delta,\,\theta}(b)}b^{S_{k-1}}\eta_k\1_{\mathcal{S}^c_k(b)}\leq 1\quad\text{a.s.}$$ The latter limit relation holds true by Lemma \ref{4} in combination with \eqref{aux011} and the fact that $\me_{k-1}(\eta_k\1_{\mathcal{S}^c_k(b)})=-\me_{k-1}(\eta_k\1_{\mathcal{S}_k(b)})$, and Lemma \ref{5}.
\end{proof}

\begin{assertion}\label{lilhalf2}
Under the assumptions of Theorem \ref{main},
\begin{equation}\label{lil11}
{\lim\sup}_{b\to 1-}\Big(\frac{1-b^2}{\log\log \frac{1}{1-b^2}}\Big)^{1/2}\sum_{k\geq 0}b^{S_k}\eta_{k+1}\geq (2{\tt s}^2 \mu^{-1})^{1/2}\quad\text{{\rm a.s.}}
\end{equation}
and
\begin{equation}\label{lil1121}
{\lim\inf}_{b\to 1-}\Big(\frac{1-b^2}{\log\log \frac{1}{1-b^2}}\Big)^{1/2}\sum_{k\geq 0}b^{S_k}\eta_{k+1}\leq -(2{\tt s}^2 \mu^{-1})^{1/2}\quad\text{{\rm a.s.}}
\end{equation}
\end{assertion}

Recall the notation: for $b\in (0,1)$ close to $1$, $$N_2(b)=\Big\lfloor \frac{1}{1-b^2}\log\frac{1}{1-b^2}\Big\rfloor.$$ Denote by $\mathbb{B}^\ast$ the class of increasing sequences $(\mb_n)_{n\in\mn}$ of positive numbers satisfying the following properties:

\noindent (a) $\lim_{n\to\infty}\mb_n=1$ and $\lim_{n\to\infty}(1-\mb_n)\log n=0$; 

\noindent (b) for large $n$, $N_{n+1}\geq N_2(\mb_n)$, where $$N_n:=\Big\lfloor\frac{\log (1-1/\log n)}{2\log \mb_n}\Big\rfloor.$$     


\noindent (c) for all $a\in (0,1)$ and some $n_0\in\mn$, $\sum_{n\geq n_0}\Big(\log\Big(\frac{1}{1-\mb_n^2}\Big)\Big)^{-a}=\infty$.

It was shown in Section 3 of \cite{Bovier+Picco:1996} (see also pp.~180,181 and 184 in \cite{Bovier+Picco:1993}) that the sequence $(\mb_n)_{n\geq 3}$ given by $$\mb_n:=\exp\Big(-\frac{1}{n!\prod_{j=2}^n (\log j)^2\prod_{k=3}^n \log\log k}\Big)$$ belongs to the class  $\mathbb{B}^\ast$.


As in the proof of Proposition \ref{lilhalf} we proceed via a sequence of lemmas.
\begin{lemma}\label{11}
Under the assumptions of Theorem \ref{main},
$$\lim_{n\to\infty} f(\mathfrak{b}_n)\sum_{k=1}^{N_n} \mb_n^{S_{k-1}}\eta_k=0\quad\text{{\rm a.s.}}$$
\end{lemma}
\begin{proof}
We start by noting that
\begin{equation}\label{aux00}
\lim_{n\to\infty}f(\mb_n) (N_n\log\log N_n)^{1/2}=0
\end{equation}
or, equivalently,
$$\lim_{n\to\infty}\frac{1-\mb_n^2}{\log\log (1/(1-\mb_n^2))}N_n\log\log N_n=0.$$ The latter is an immediate consequence of $(1-\mb_n^2)N_n\sim (\log n)^{-1}\to 0$ as $n\to\infty$ and $${\lim\sup}_{n\to\infty}\frac{\log\log N_n}{\log\log (1/(1-\mb_n^2))}\leq 1.$$

Formula \eqref{summat} with $\ell=N_n$ and $b=\mb_n$ reads $$\sum_{k=1}^{N_n} \mb_n^{S_{k-1}}\eta_k=\sum_{k=1}^{N_n-1}(\mb_n^{S_{k-1}}-\mb_n^{S_k})T_k +\mb_n^{S_{N_n-1}}T_{N_n},$$ where $T_k=\eta_1+\ldots+\eta_k$ for $k\in\mn$. Using $\lim_{n\to\infty}N(n)\log \mb_n=0$ in combination with the strong law of large numbers we infer
\begin{equation}\label{slln}
\lim_{n\to\infty}\mb_n^{S_{N_n-1}}=1\quad\text{and}\quad \lim_{n\to\infty}\mb_n^{S_{N_n}}=1\quad\text{a.s.}
\end{equation}
This together with the law of the iterated logarithm for standard random walks entails $${\lim\sup}_{n\to\infty}\frac{\mb_n^{S_{N_n-1}}T_{N_n}}{(N_n\log\log N_n)^{1/2}} =2^{1/2}\quad\text{a.s.}$$
Further, with the same $N$ as in \eqref{aux53} (we replace $\mu-\varepsilon$ with $\delta$),
$$\lim_{n\to\infty}\sum_{k=1}^N (\mb_n^{S_{k-1}}-\mb_n^{S_k})T_k=0\quad\text{a.s.}$$ According to \eqref{aux53}, for large enough $n$ and a constant $c>0$,
\begin{multline*}
\Big|\sum_{k=N+1}^{N_n-1}(\mb_n^{S_{k-1}}-\mb_n^{S_k})T_k\Big|\leq \sum_{k=1}^{N_n-1}|\mb_n^{S_{k-1}}-\mb_n^{S_k}||T_k|\leq c |\log \mb_n|\sum_{k\geq 1} \mb_n^{\delta k}|\xi_k|(\sup_{k\leq N_n}\,|T_k|)\\=O((N_n\log\log N_n)^{1/2})\quad\text{a.s.}
\end{multline*}
The last equality is a consequence of Theorem \ref{main1} (which gives $\lim_{n\to\infty}|\log \mb_n|\sum_{k\geq 1} \mb_n^{\delta k}|\xi_k|=\delta^{-1}\me |\xi|$ a.s.) and \eqref{lilweak}.
An appeal to \eqref{aux00} completes the proof of Lemma \ref{11}.
\end{proof}

\begin{lemma}\label{aux200}
Under the assumptions of Theorem \ref{main}, for all $\varepsilon\in (0,1)$,
\begin{equation}\label{bkimpo1}
\mmp\Big\{f(\mb_n)\sum_{k=N_n+1}^{N_2(\mb_n)} \mb_n^{S_{k-1}}\eta_k >1-\varepsilon\quad \text{{\rm i.o.}}\Big\}=1.
\end{equation}
\end{lemma}
\begin{proof}
Assume that we have already proved that, for all $\varepsilon_1\in (0,1)$,
\begin{equation}\label{bk}
\mmp\{\mathcal{C}_n(\varepsilon_1)~\text{{\rm i.o.}}\}=1,
\end{equation}
where
$$\mathcal{C}_n(\varepsilon_1):=\Big\{f(\mb_n)\mb_n^{-S_{N_n}}\sum_{k=N_n+1}^{N_2(\mb_n)} \mb_n^{S_{k-1}}\eta_k >1-\varepsilon_1\Big\},\quad n\in\mn.$$
Setting, for each $\varepsilon_2\in (0,1)$, $\mathcal{D}_n(\varepsilon_2):=\{\mb_n^{S_{N_n}}>1-\varepsilon_2\}$ we conclude with the help of the second equality in \eqref{slln} that, for all $\varepsilon_2\in (0,1)$,  $\mmp\{\mathcal{D}_n(\varepsilon_2)~\text{eventually}\}=1$. This in combination with \eqref{bk} yields, for all $\varepsilon_1,\varepsilon_2\in (0,1)$, $$\mmp\Big\{\mathcal{C}_n(\varepsilon_1)\bigcap\mathcal{D}_n(\varepsilon_2)~\text{{\rm i.o.}}\Big\}=1.$$ Since
$$\mathcal{C}_n(\varepsilon_1)\bigcap\mathcal{D}_n(\varepsilon_2)\subseteq \Big\{f(\mb_n)\sum_{k=N_n+1}^{N_2(\mb_n)} \mb_n^{S_{k-1}}\eta_k >(1-\varepsilon_1)(1-\varepsilon_2)\Big\},\quad n\in\mn,$$ we arrive at \eqref{bkimpo1}.

By the property (b) of $\mathbb{B}^\ast$, $N_{n+1}\geq N_2(\mb_n)$ for large $n$ which implies that, for large $n$, the random variables $$\mb_n^{-S_{N_n}}\sum_{k=N_n+1}^{N_2(\mb_n)}\mb_n^{S_{k-1}}\eta_k=\eta_{N_n+1}+\mb_n^{\xi_{N_n}+1}\eta_{N_n+2}+\ldots+\mb_n^{\xi_{N_n+1}+\ldots+\xi_{N_2(\mb_n)-1}}\eta_{N_2(\mb_n)} 
$$ are independent. Hence, by the converse part of the Borel-Cantelli lemma, \eqref{bk} is a consequence of
\begin{equation}\label{eq:aux}
\sum_{n\geq 1}\mmp\Big\{f(\mb_n)\mb_n^{-S_{N_n}}\sum_{k=N_n+1}^{N_2(\mb_n)} \mb_n^{S_{k-1}}\eta_k >1-\varepsilon\Big\}=\infty.
\end{equation}

We intend to prove \eqref{eq:aux}. Fix any $\delta\in (0,1)$. For each $n\in\mn$ and $t>0$, put $$q_n(t):=2t^{-2}\Big \lfloor \log\log \frac{1}{1-\mb_n^2}\Big\rfloor.$$ For notational simplicity, we shall write $q_n$ for $q_n(t)$. Further, for each $n\in\mn$ and each nonnegative integer $k\leq q_n$ define numbers $r_{k,n}$ by $r_{0,n}:=0$, $$r_{k,n}:=\inf\Big\{j\geq r_{k-1,n}+1: \mb_n^{2\delta r_{k-1,n}}\sum_{k=r_{k-1,n}}^{j-1} \mb_n^{2k} \geq \sigma_n^2q_n^{-1}\Big\},\quad k\in\mn, k\leq q_n-1,$$ where $\sigma_n^2:=\sum_{k=0}^{N_2(\mb_n)-N_n-1}\mb_n^{2k}$, and $r_{q_n,n}:=N_2(\mb_n)-N_n+1$. One can check by a direct calculation that the numbers are well-defined and that, for $k\in\mn$, $k\leq q_n$, 
\begin{equation}\label{eq:aux2}
\sum_{k=0}^{r_{k,n}-r_{k-1,n}-1}\mb_n^{2k}~\sim~r_{k,n}-r_{k-1,n}~\sim~ \frac{\sigma_n^2}{q_n},\quad n\to\infty.
\end{equation}
For the latter we have used the fact that the relations $\lim_{n\to\infty}\mb_n^{2N_n}=1$ and $\lim_{n\to\infty}\mb_n^{2N_2(\mb_n)}=0$ entail $$\sigma_n^2~\sim~(1-\mb_n^2)^{-1},\quad n\to\infty.$$ For each $n\in\mn$, $k\in\mn_0$, $k\leq q_n$ and $t>0$, put
$$\tilde Z_{k,n}:=q_n^{1/2}\sigma_n^{-1}\sum_{j=r_{k-1,n}}^{r_{k,n}-1}\mb_n^{S_j}\eta_{j+1}$$ and $$Z_{k,n}:=q_n^{1/2}\sigma_n^{-1}\mb_n^{(1+\delta)r_{k-1,n}}\sum_{j=r_{k-1,n}}^{r_{k,n}-1}\mb_n^{S_j-S_{r_{k-1,n}}}\eta_{j+1}.$$ Observe that the random variables $Z_{1,n},\ldots, Z_{q_n,n}$ are independent, and $$Z_{k,n}\overset{{\rm d}}{=} q_n^{1/2}\sigma_n^{-1}\mb_n^{(1+\delta)r_{k-1,n}}\sum_{j=0}^{r_{k,n}-r_{k-1,n}-1}\mb_n^{S_j}\eta_{j+1},$$ where $\overset{{\rm d}}{=}$ denotes equality of distributions. Noting that $\lim_{n\to\infty}\mb_n^{(1+\delta)r_{k-1,n}}=1$ and then using \eqref{eq:aux2} we infer with the help of Lemma \ref{lem:aux} that
\begin{equation}\label{eq:aux1}
\lim_{n\to\infty}\mmp\{Z_{k,n}\leq x\}=\mmp\{{\rm Normal}\,(0,1)\leq x\},\quad x\in\mr.
\end{equation}

In view of $$f(\mb_n)~\sim~(\sigma_n q_n^{1/2}t)^{-1} 
:=\alpha_n,\quad n\to\infty,$$ it suffices to prove \eqref{eq:aux} with $\alpha_n$ replacing $f(\mb_n)$. Then with $r_{-1,n}:=0$
\begin{multline*}
\mmp\Big\{\alpha_n\mb_n^{-S_{N_n}}\sum_{k=N_n+1}^{N_2(\mb_n)} \mb_n^{S_{k-1}}\eta_k >1-\varepsilon\Big\}=\mmp\Big\{\alpha_n \sum_{k=0}^{N_2(\mb_n)-N_n-1} \mb_n^{S_k}\eta_{k+1}>1-\varepsilon\Big\}\\=\mmp\Big\{\sigma_n q^{1/2}_n t \alpha_n q_n^{-1} \sum_{k=1}^{q_n} \tilde Z_{k,n}>t(1-\varepsilon)\Big\}\geq \mmp\Big\{\tilde Z_{k,n}>t(1-\varepsilon), 1\leq k\leq q_n\Big\}\\ \geq \mmp\Big\{\tilde Z_{k,n}>t(1-\varepsilon), S_{r_{k-1,n}}-S_{r_{k-2,n}}\leq (1+\delta)(r_{k-1,n}-r_{k-2,n}), 1\leq k\leq q_n\Big\}\\\geq \mmp\Big\{Z_{k,n}>t(1-\varepsilon), S_{r_{k-1,n}}-S_{r_{k-2,n}}\leq (1+\delta)(r_{k-1,n}-r_{k-2,n}), 1\leq k\leq q_n\Big\}\\=
\prod_{k=1}^{q_n}\mmp\Big\{Z_{k,n}>t(1-\varepsilon), S_{r_{k-1,n}}-S_{r_{k-2,n}}\leq (1+\delta)(r_{k-1,n}-r_{k-2,n})\Big\}.
\end{multline*}
Using \eqref{eq:aux1} and the weak law of large numbers for random walks we conclude that, uniformly in $k\in\mn$, $k\leq q_n$,
\begin{equation}\label{eq:aux3}
\lim_{n\to\infty} \mmp\Big\{Z_{k,n}>t(1-\varepsilon), S_{r_{k-1,n}}-S_{r_{k-2,n}}\leq (1+\delta)(r_{k-1,n}-r_{k-2,n})\Big\}=\mmp\{{\rm Normal}\,(0,1)> t(1-\varepsilon)\}.
\end{equation}
Given constants $c\in (0,1)$ and $\rho\in (0,1)$ we can choose $t$ so large that $$A(t):=2t^{-2}(\log (1/c)+\rho+\log (t(1-\varepsilon)))+(1-\varepsilon)^2<1$$ and that, for large $n$,
\begin{multline*}
\log\mmp\Big\{Z_{k,n}>t(1-\varepsilon), S_{r_{k-1,n}}-S_{r_{k-2,n}}\leq (1+\delta)(r_{k-1,n}-r_{k-2,n})\Big\}\\\geq \log \mmp\{{\rm Normal}\,(0,1)> t(1-\varepsilon)\} -\rho\geq -(\log (1/c)+\rho+2^{-1}t^2(1-\varepsilon)^2 +\log (t(1-\varepsilon)))=2^{-1}t^2 A(t),
\end{multline*}
where the first inequality is a consequence of \eqref{eq:aux3}, and the second inequality follows from Lemma 12.9 on p.~349 in \cite{Peres+Moerters:2010}. Hence, for $c$, $\rho$, $t$ and $n$ as above
$$\mmp\Big\{\alpha_n\mb_n^{-S_{N_n}}\sum_{k=N_n+1}^{N_2(\mb_n)} \mb_n^{S_{k-1}}\eta_k >1-\varepsilon\Big\}\geq \exp(-2^{-1}t^2A(t)q_n)=\exp\Big(-A(t)\lfloor\log\log(1/(1-\mb_n^2)) \rfloor\Big).$$ This is the general term of a divergent series, hence \eqref{eq:aux} holds, because $$\sum_{n\geq n_0}\Big(\log \frac{1}{1-\mb_n^2}\Big)^{-A(t)}=\infty$$ by the property (c) of $\mathbb{B}^\ast$. The proof of Lemma \ref{aux200} is complete.
\end{proof}

Now we can prove Proposition \ref{lilhalf2} and Theorem \ref{main}.

\begin{proof}[Proof of Proposition \ref{lilhalf2}]
Relation \eqref{lil11} is a consequence of formula \eqref{aux121} 
and Lemmas \ref{11} and \ref{aux200}. Replacing in \eqref{lil11} $\eta_k$ with $-\eta_k$ we obtain \eqref{lil1121}.
\end{proof}

\begin{proof}[Proof of Theorem \ref{main}]
Relation \eqref{auxlimsup} follows from Propositions \ref{lilhalf} and \ref{lilhalf2}.

Recalling our convention that $\mu={\tt s}^2=1$ it remains to prove that
\begin{equation}\label{limit points}
C\bigg(\bigg(f(b)\sum_{k\geq 0}b^{S_k}\eta_{k+1}: b\in ((1-e^{-1})^{1/2},1)\bigg)\bigg)=[-1,1]\quad\text{{\rm a.s.}}
\end{equation}
To this end, we first note that the random function $b\mapsto \sum_{k\geq 1}b^{S_{k-1}}\eta_k$ is a.s.\ continuous on $[0,1)$. Indeed, while the function $b\mapsto b^{S_{k-1}}\eta_k$ is a.s.\ continuous on $[0,1)$, the latter series converges uniformly on $[0,a]$ for each $a\in (0,1)$ with probability one. This follows from the inequality $b^{S_{k-1}}\leq b^{\delta(k-1)}\leq a^{\delta(k-1)}$ which holds for large $k$ and $b\in [0,a]$ and the fact that $\me \sum_{k\geq 1} a^{\delta(k-1)}|\eta_k|<\infty$. Thus, the function $b\mapsto f(b)\sum_{k\geq 1}b^{S_{k-1}}\eta_k $ is a.s.\ continuous on $((1-e^{-1})^{1/2}, 1)$ with ${\lim\sup}_{b\to 1-}=1$ and ${\lim\inf}_{b\to 1-}=-1$. This immediately entails \eqref{limit points} with the help of the intermediate value theorem for continuous functions.
\end{proof}

\noindent {\bf Acknowledgement}. A. Iksanov and I. Samoilenko were supported by the National Research Foundation of Ukraine (project 2020.02/0014 ``Asymptotic regimes of perturbed random walks: on the edge of modern and classical probability''). A. Iksanov thanks Alexander Marynych for a useful discussion concerning the proof of Lemma \ref{3}.

\end{document}